 \documentclass[11pt,twoside,reqno,centertags,draft]{amsart}


 \usepackage{amsmath,amsthm,amsfonts,amssymb}
\usepackage[mathscr]{euscript}
 \usepackage{hyperref}
 \usepackage{mathrsfs}
\usepackage{graphicx}
\usepackage{color}
  \usepackage{epsfig}

\usepackage{mathtools}
\usepackage{comment}

\newtheorem{theorem}{Theorem}[section]

\newtheorem{coro}{Corollary}[section]
\newtheorem{definition}{Definition}[section]
\newtheorem{lem}{Lemma}[section]

\newtheorem{remark}{Remark}[section]

\newcommand{\be}{\begin{eqnarray}}
\newcommand{\ee}{\end{eqnarray}}

\newcommand{\R}{\mathbb{R}}
\newcommand{\C}{\mathbb{C}}
\newcommand{\N}{\mathbb{N}}

\def\e{{\rm e}}

\begin{document}

\title{ Generalized random processes related to Hadamard operators and Le~Roy measures }
\date{}
\maketitle

\vspace{ -1\baselineskip}

{\small
\begin{center}
 {\sc Luisa Beghin}\footnote{Corresponding author}\\
 Department of Statistical Sciences, Sapienza University of
Rome \\
p.le A.Moro 5, 00185, Rome, Italy\\[10pt]
{\sc Lorenzo Cristofaro}\\
Department of Mathematics, Luxembourg University\\
 6 avenue de la Fonte,
L-4364, Esch-sur-Alzette, Luxembourg\\[10pt]
{\sc Federico Polito}\\
Department of Mathematics ``G.~Peano'', University of Torino\\
via Carlo Alberto 10, 10123, Turin, Italy \\[10pt]
 
\end{center}
}

\numberwithin{equation}{section}
\allowdisplaybreaks

 \smallskip

 \begin{quote}
\footnotesize
{\bf Abstract.}
The definition of generalized random processes in
Gel'fand sense allows to extend well-known
stochastic models, such as the fractional
Brownian motion, and study the related fractional
pde's, as well as stochastic differential
equations in distributional sense.
By analogy with the construction
(in the infinite-dimensional white-noise space)
of the latter, we introduce two processes
defined by means of Hadamard-type fractional
operators. When used to replace the time
derivative in the governing p.d.e.'s,
the Hadamard-type derivatives are usually
associated with ultra-slow diffusions.
On the other hand, in our construction,
they directly determine
the memory properties of the so-called
Hadamard fractional Brownian motion (H-fBm)
and its long-time behaviour. Still,
for any finite time horizon, the H-fBm
displays a standard diffusing feature.
We then extend the definition of the
H-fBm from the white noise space to an
infinite dimensional grey-noise space
built on the Le Roy measure, so that our
model represents an alternative to the
generalized grey Brownian motion.
In this case, we prove that the
one-dimensional distribution of the
process satisfies a heat equation with
non-constant coefficients and fractional
Hadamard time-derivative.
Finally, once proved the existence of
the distributional derivative
of the above defined processes and derived
an integral formula for it,
we construct an Ornstein-Uhlenbeck type
process and evaluate its distribution.

\end{quote}

\section{Introduction}
During the XX century developments in harmonic analysis lead to
the definition of infinite-dimensional linear topological spaces
(e.g.~nuclear spaces, Gel'fand triples), whose impact in analysis
and probability theory was extremely important, see
\cite{GV64}. Indeed, Bochner-Minlos theorem allows to define
probability spaces, through Gaussian or non-Gaussian measures on
such infinite-dimensional spaces, where the so-called generalized
random processes or fields exploit the notion of random variable
through distributional functionals, see
\cite{HID,SCHN}.
In this setting, the white noise space is a Gaussian space where
the random variables are indeed generalized random processes (in
Gel'fand sense), expressed by the action of a tempered generalized
function,
$\omega$, on an element of the Schwartz space of test functions
$\mathcal{S}(\mathbb{R})$. Thanks to density of Schwartz functions
in the space of square integrable functions, the latter could be
used so that
$\omega(f)=\langle \omega, f \rangle$ is a centered Gaussian
variable with variance
$\| f\|$, where
$\|\cdot\|$ is the norm of
$L^2(\R)$.
This framework enables to define well-known stochastic processes
(e.g.~Brownian motion or fractional Brownian motion (fBm)) by
choosing a specific square integrable function, or properly define
stochastic processes' derivatives in the distributional sense.

The fBm can be defined, as a generalized stochastic process, on
the white noise space
$(\mathcal{S}'(\mathbb{R}), \mathcal{B},\nu)$, where
$\mathcal{S}'(\mathbb{R})$ is the dual of
$\mathcal{S}(\mathbb{R})$,
$\mathcal{B}$ is the cylinder
$\sigma$-algebra  and
$\nu(\cdot)$ is the white noise (Gaussian) measure, as follows
(see e.g.\
\cite{GRO2}):
\begin{equation}
B_\alpha (t,\omega)
:=\langle \omega,
\mathcal{M}^{\alpha/2} _{-}1_{[0,t)}
\rangle,\qquad t\geq 0,\;\omega \in
\mathcal{S}^{\prime }
(\mathbb{R}) \label{fBm}
\end{equation}
where
\begin{equation}\label{mrie}
 ( \mathcal{M}_{-}^{\alpha/2 }fv) (x)
:=\left\{
\begin{array}{l}
C_\alpha
( \mathcal{D}_{-}^{ { \frac{1-\alpha}{2} }  }f ) (x),\qquad
\alpha \in (0,1)
\\
f(x),\qquad \qquad \qquad \qquad \; \alpha =1 \\
C_\alpha  ( \mathcal{I}_{-}^{ { \frac{\alpha-1 }{2} }  }f ) (x),\qquad
\alpha \in (1,2)
\end{array}
\right.
\end{equation}
$C_\alpha= \sin(\pi \alpha/2) \Gamma(1+\alpha)$ and
$ \mathcal{D}_{-}^{\gamma }$ (resp.\
$ \mathcal{I}_{-}^{\gamma }$) is the Riemann-Liouville right-sided
fractional derivative (resp. integral) of order
$\gamma$ (see
\cite{KIL} pp.\ 79-80, for their definitions).

The above definition in white noise space makes the wavelet
decomposition possible, as well as the consequent stochastic
integral representation permits moving average or harmonizable
representations of the fractional Brownian motion; see
\cite{AS96, MST99}. Indeed, the latter were extensively used in
applications such as time series analysis, spectrum study and in
order to introduce complex-order fractional operators for
whitening; see
\cite{AFU24, BDS21}.

The model defined in \eqref{fBm} has been extended to the
so-called
\emph{generalized grey Brownian motion} (hereafter ggBm),
by considering the definition (\ref{fBm}) in the space
$(\mathcal{S}'(\mathbb{R}), \mathcal{B},\nu_{\rho})$, where
$\nu_\rho$, for
$\rho \in (0,1)$, is the Mittag-Leffler measure, i.e.,
the unique measure satisfying
\[
\int_{\mathcal{S}'(\mathbb{R})} e^{i \langle \omega , \xi \rangle}
d\nu_{\rho}(\omega)=E_\rho
\left ( - \tfrac{1}{2} \langle \xi , \xi
\rangle
\right ), \qquad \xi \in \mathcal{S}(\mathbb{R}),
\]
where
$$
E_\rho
\left ( x
\right )
:=\sum_{j=0}^{\infty}\frac{x^j}{\Gamma(\rho
j+1)} ,
$$
 is the Mittag-Leffler function.
Indeed, the latter is the eigenfunction of the
left-sided Caputo-type fractional derivative
$ D_{+}^{\rho}$, i.e.,
\begin{equation}
D_{0+}^{\rho} E_\rho(\lambda t^\rho)=\lambda E_\rho(\lambda
t^\rho), \qquad t\geq 0, \lambda \in \mathbb{R}\label{mittlef}
\end{equation}
(see
\cite{KIL}, p.\ 98).
It is well-known that the ggBm
$B_{\alpha,\rho}
:=\left\{B_{\alpha,\rho} (t)\right\}_{t \geq 0}$
is a non-Gaussian process with zero mean and covariance function
\[
cov(B_{\alpha,\rho}(t),B_{\alpha,\rho}(s))=
\frac{1}{2\Gamma(1+\rho)}(t^{\alpha}+
s^{\alpha}-|t-s|^{\alpha}), \qquad s,t \in \mathbb{R}^+.
\]
Thus the ggBm has non-stationary increments and it is an anomalous
diffusion, since
$\mathbb{E}(B_{\alpha,\rho}(t))^2 \sim c_{\rho}t^\alpha$, where
$c_{\rho}
:=1/\Gamma(\rho+1)$. Moreover,
it displays short- (resp.\ long-) range dependence for
$\alpha \in (0,1)$ (resp.\
$\alpha \in (1,2)$) as the fBm. An alternative model has been
constructed in
\cite{BEG} and applied in
\cite{CRI}, by means of the so-called incomplete gamma measure.
Further extensions are considered in
\cite{ALP}. For definitions of different processes, on infinite
dimensional spaces, based on Poisson and Gamma measures, see also
\cite{KdSSU98}.

Our first aim is to follow a similar procedure in order to define
an analogue of the fBm (and of the ggBm) by substituting the
(right-sided) Riemann-Liouville operators by their Hadamard
counterparts, i.e.,
$^H\mathcal{D}_-^\gamma$  and
$^H\mathcal{I}_-^\gamma$, for
$\gamma= { \frac{1-\alpha}{2} } $ and
$\gamma= { \frac{\alpha-1 }{2} } $, respectively (see \eqref{dermin} and
\eqref{intmin} below with
$\mu=0$). Therefore, in our case, we will define,
 in the white-noise space, the
 \emph{Hadamard fractional Brownian motion}
 (hereafter H-fBm) as
$B^H_\alpha
:=\left\{ B^H_\alpha (t)\right\}_{t \geq 0}$, where
$B^H_\alpha (t,\omega)
:=\langle \omega,
 \thinspace ^H\mathcal{M}^{\alpha/2} _{-}1_{[0,t)} \rangle$,
$t\geq 0,\;\omega \in \mathcal{S
}^{\prime }(\mathbb{R})$, and
\begin{equation*}
 ( \thinspace ^H\mathcal{M}_{-}^{\alpha/2 }f )
(x)
:=\left\{
\begin{array}{l}
K_\alpha ( \thinspace ^H\mathcal{D}_{-}^{ { \frac{1-\alpha}{2} }
}f )
(x),\qquad \alpha \in (0,1)
\\
f(x),\hspace{3.6cm} \alpha =1 \\
K_\alpha  ( \thinspace ^H\mathcal{I}_{-}^{ { \frac{\alpha-1 }{2} }
}f )
(x),\qquad \alpha \in (1,2).
\end{array}
\right.
\end{equation*}

The Hadamard fractional derivatives are usually associated to
ultra-slow diffusions (i.e.\ with mean-squared displacement given
by a logarithmic
function of time); see, for example,
\cite{LIA}.
On the other hand, we prove that, in our construction, the H-fBm
is a (centered, Gaussian) process with
$var(B^H_\alpha (t))=t$; thus its one-dimensional distribution
coincides with that of a standard Brownian motion, for any
$\alpha$, and hence the Hadamard operator does not
affect the one-dimensional distribution. Nevertheless, the
parameter
$\alpha$ affects its auto-covariance (expressed in terms of
Tricomi’s confluent hypergeometric function), as
 well as its long-time behaviour. Indeed,
$B^H_\alpha$ presents anti-persistent or long-range dependent
increments, for
$\alpha \in (0,1)$ and
$\alpha \in (1,2)$, respectively.

We also give the following, finite-dimensional, representation of
H-fBm,  in terms of a stochastic integral, which is the analogue
of the Mandelbrot-Van Ness representation for the fBm:
  \begin{equation}
      B_{\alpha}^H (t)=\frac{1}{\sqrt{\Gamma(\alpha)}}\int
_{0}^{t}
\Big ( \log \frac{t}{s}
\Big )^{ { \frac{\alpha-1 }{2} } }dB(s), \qquad t
\geq 0, \quad \alpha \in (0,2), \notag
  \end{equation}
  where
$\left\{B(t)\right\}_{t \geq 0}$ is a standard Brownian motion.

There has recently been considerable interest in models which can
be seen as intermediate between the standard Brownian motion and
the fBm: for example, the so called ``sub-fractional Brownian
motion" (or sub-fBm) was introduced in
\cite{BOJ} and further studied in
\cite{TUD},
\cite{ARA}.
Similarly to the sub-fBm, our model enjoys the main properties of
the fBm, such as long-range dependence (in the range
$(1,2)$ of the parameter
$\alpha$), self-similarity and continuous sample paths, but it has
non-stationary increments. Moreover, in our case, the variance
coincides with that of the standard Bm and the long-range
dependence turns out to be weaker than for the fBm.

Diffusion processes such as the H-fBm could be considered as
potential candidates to model some
financial time-series which exhibit standard mean square
displacement (i.e.,
$var(B_\alpha^H(t))=t$), but with long-range
dependence and non-stationary increments.

We then extend the definition of the H-fBm to the space
$(\mathcal{S}'(\mathbb{R}), \mathcal{B},\nu_\beta)$, where
$\nu_\beta$ is the unique measure satisfying
\begin{equation}
\int_{\mathcal{S}^{\prime }}
e^{i\left\langle x,\xi \right\rangle }d\nu
_{\beta }(x)=\mathcal{R}_{\beta }
\Big ( -\frac{\langle \xi ,\xi \rangle }{2}
\Big ) \text{,\qquad }\xi \in \mathcal{S},
\end{equation}
and
$$
\mathcal{R}_\beta (x)
:=\sum_{j=0}^{\infty}x^j /(j!)^\beta ,
 \ \text{ for
$\beta \in (0,1]$,}
$$
 is the Le Roy function (see
\cite{ler},
\cite{BOU}, for details, and
\cite{GER},
\cite{gar2},
\cite{SIM} for recent generalizations).
This choice is motivated by the fact that
$\mathcal{R}_{\beta }(t)$ satisfies the following equation
\begin{equation}
{}^{H}D_{0+}^{\beta }f(t)=t
f(t), \qquad t \geq 0, \notag
\end{equation}
where
$^{H}D_{0+}^{\beta }$ is the left-sided Hadamard derivative of
Caputo type of order
$\beta$ (see \eqref{cap} below); cf.\ equation \eqref{mittlef} in
the Mittag-Leffler case.
For the Le Roy measure
$\nu_\beta$ we prove the existence of test functions and we
establish the characterization theorems for the corresponding
distribution space, after checking the Le Roy measure is analytic
on
$(\mathcal{S}', \mathcal{B})$ and its Laplace transform is
holomorphic.

In this case, we term the corresponding process
 \emph{Le Roy-Hadamard motion} (LHm for brevity)
 and denote it by
$$
B^H_{\alpha,\beta}
:= \{ B^H_{\alpha,\beta}
(t) \}_{t\geq 0} .
$$
 In the limiting case
$\beta=1$,
$B^H_{\alpha,\beta}$ and
$B^H_{\alpha}$ coincide as
the Le Roy
function reduces to the exponential function.

As we will see below, since the Le Roy measure's two moments do
not depend on
$\beta$, the covariance function (and thus the persistence and
long-range properties) of
$B^H_{\alpha,\beta}$ coincides, for any
$\beta$, with that of
$B^H_{\alpha}$.
Moreover, we prove that the one-dimensional distribution of
$B^H_{\alpha,\beta}$ satisfies the following fractional heat
equation with non-constant coefficients:
\begin{equation}
^{H}D_{0+,t}^{\beta }u(x,t)=\frac{ t}{2}\frac{\partial
^{2}}{\partial x^{2}}u(x,t), \qquad x \in \mathbb{R}, t \geq 0,
\notag
\end{equation}
with initial condition
$u(x,0)=\delta (x),$ where
$\delta (\cdot
)$ is the Dirac's delta function. This result can be compared with
the master equation, which was proved in
\cite{MUR} to be satisfied by the one-dimensional distribution of
the ggBm, and later generalized in
\cite{BEN}.

We prove the existence and derive an integral formula for the
distributional derivative of the LHm, by evaluating the
$S_{\nu_\beta}$-transform of
$B^H_{\alpha,\beta}$ and of its noise.
These results can be considered as the basis for constructing a
stochastic analysis theory driven by the LHm, by following a
Wick-type
definition of stochastic differential equations similar to
the one applied in
\cite{BOC} for the ggBm (see also
\cite{BOC2} for the vector-valued ggBm).

Finally, as a further application of the latter results, we define
an Ornstein-Uhlenbeck type process based on the LHm and evaluate
its distribution.

\section{Preliminary results}
We recall that the \textit{left-sided} and \textit{
right-sided Hadamard-type integral} are defined,
respectively, as:
\begin{align}
(^{H}\mathcal{I}_{0+, \mu}^{\gamma }f)(t)
 &
:=\frac{1}{\Gamma (\gamma )}
\int_{0}^{t}
\Big (\frac{z}{t}
\Big )^\mu
\Big ( \log \frac{t}{z}
\Big ) ^{\gamma -1}
\frac{f(z)}{z}dz,
\label{hadint} \\
(^{H}\mathcal{I}_{-, \mu}^{\gamma }f)(t)
&
:=\frac{1}{\Gamma (\gamma )}
\int_{t}^{\infty }
\Big (\frac{t}{z}
\Big )^\mu
\Big ( \log \frac{z}{t}
\Big ) ^{\gamma -1}\frac{f(z)}{z}
dz,  \label{intmin}
\end{align}
for
$t>0$,
$\gamma, \mu \in \mathbb{C},$
$\Re(\gamma )>0$, where
$\Re(\cdot)$ denotes the real part (see
\cite{KIL}, equations (2.7.5)-(2.7.6)).
The \textit{left-sided Hadamard-type derivative}
of order
$\gamma \geq 0$ and parameter
$\mu \in \mathbb{C}$ is
defined as
\begin{equation}
(^{H}\mathcal{D}_{0+, \mu}^{\gamma }f)(t)
:=t^{-\mu}
\Big ( t
\frac{d}{dt}
\Big ) ^{n}\left[t^\mu (^{H}
\mathcal{I}_{0+, \mu}^{n-\gamma }f)(t)\right],
 \label{had}
\end{equation}
while the \textit{right-sided
Hadamard-type derivative}
 is given by
\begin{equation}
(^{H}\mathcal{D}_{-, \mu}^{\gamma }f)
(t)
:=t^{\mu}
\Big ( -t
\frac{d}{dt}
\Big ) ^{n}\left[t^{-\mu}
(^{H}\mathcal{I}_{-, \mu}^{n-\gamma }f)(t)\right],
   \label{dermin}
\end{equation}
where
$\gamma \notin \mathbb{N}$,
$\Re(\gamma)>0$,
$n=\left\lfloor \gamma \right\rfloor +1$ and
$t>0$ (see (2.7.11) and
(2.7.12) in
\cite{KIL}, for
$a=0$ and
$b=\infty )$. When
$\gamma =m$, for
$m \in \mathbb{N}$,
$$
(^{H}\mathcal{D}_{0+, \mu}^{\gamma }f)(t)
:=t^\mu
\Big ( t
\frac{d}{dt}
\Big ) ^{m}(t^{-\mu}f(t))
$$
 and
$$
(^{H}\mathcal{D}_{-, \mu}^{\gamma }f)(t)
:=(-1)^mt^\mu
\Big ( t
\frac{d}{dt}
\Big ) ^{m}(t^{-\mu}f(t)) .
$$
 Hereafter, we will write for brevity, in the case
$\mu=0$,
$^{H}\mathcal{I}_{0+}^{\gamma }
:=\thinspace^{H}\mathcal{I}_{0+, 0}^{\gamma }$,
$^{H}\mathcal{I}_{-}^{\gamma }
:=\thinspace^{H}\mathcal{I}_{-, 0}^{\gamma }$,
$^{H}\mathcal{D}_{0+}^{\gamma }
:=\thinspace^{H}\mathcal{D}_{-, 0}^{\gamma }$ and
$^{H}\mathcal{D}_{-}^{\gamma }
:=\thinspace^{H}\mathcal{D}_{-, 0}^{\gamma }$.

\begin{remark}\rm
\label{rem:AbsoutelycontinousSchwartz}
The operators introduced above are well
defined in the space
\[
X_\mu^p
:=\Big \{ h:
\Big (\int_0 ^{\infty}
|z^\mu h(z)|^p \frac{dz}{z}
\Big )^{1/p}
< \infty , p \in [1, \infty),
 \mu \in \mathbb{R} \Big \},
\]
which, for
$\mu=1/p$ reduces to the well-known
$L^p (\mathbb{R}^+)$ (see
\cite{BUT} and
\cite{KIL} for more details).
    For
$\gamma \in (0,1)$,
$\mu=0$, the domain of the above
left-sided Hadamard operators contains
$AC[0,T]$ (see
\cite{KIL}, p.\ 3). In view of what
follows, we note that a Schwartz function
$\xi(\cdot)$ can be embedded in the
space of absolutely continous functions as it holds
$\|\xi \|_{AC[0,T]}
\leq \|\xi\|_{0,0}+T\|\xi\|_{0,1}$,
where
$\|\cdot \|_{AC[0,T]}$ is the norm of
$AC[0,T]$ and
$\{\|\cdot \|_{r,s}, r,s \in \N \}$
is the family of norms of the Schwartz space
$\mathcal{S}(\R)$.
\end{remark}

We also recall the
\textit{left and right-sided Hadamard
derivative of Caputo type} of order
$
\gamma \in (0,1)$,
which are respectively defined as follows:
\begin{eqnarray}
(^{H}D_{0+}^{\gamma }f)(t) &
:=&\frac{1}{\Gamma (1-\gamma
)}\int_{0}^{t}
\Big ( \log \frac{t}{z}
\Big ) ^{-\gamma
}\frac{d}{dz}f(z)dz,  \label{cap} \\
(^{H}D_{-}^{\gamma }f)(t) &
:=&-\frac{1}{\Gamma (1-\gamma )}\int_{t}^{\infty }
\Big ( \log \frac{z}{t}
\Big ) ^{-\gamma}\frac{d}{dz}
f(z)dz,
\end{eqnarray}
(see
\cite{GAR}; the relationship between
$^{H}D_{a+}^{\gamma }$ and
$^{H}\mathcal{D}_{a+}^{\gamma }$ is given in
\cite{KIL03}, Theorem 3.2, for
$a>0$).

Finally, in the last section of the paper,
we will apply the following relationship
\begin{equation}
^{H}\mathcal{D}_{0+,\mu}^{\gamma}f
\equiv \thinspace^H\mathbb{D}_{0+,\mu}^\gamma f,
   \label{MH}
\end{equation}
which holds for any
$f \in X_c^p$ between
$\thinspace ^{H}\mathcal{D}_{0+,\mu}^{\gamma}$,
given in \eqref{had}, and the left-sided
Marchaud-Hadamard type derivative
\begin{align}
 \label{MH2}
& (\thinspace^H\mathbb{D}_{0+,\mu}^\gamma
f )(x) \\
\notag
:\! &=  \frac{\gamma}{\Gamma(1-\gamma)}
\lim_{\epsilon \to 0^+} \int^x _{\epsilon}
\left (\frac{z}{x}
\right )^\mu
\left (\log\frac{x}{z}
\right )^{-\gamma -1}
\left[f(x)-f(z)\right]
\frac{dz}{z}+\mu^\gamma f(x), \notag  \\
&= \frac{\gamma}{\Gamma(1-\gamma)}
\lim_{\epsilon \to 0^+}
\int^{\infty} _\epsilon e^{-\mu z}
\frac{f(x)-f(xe^{-z})}{z^{1+\gamma}}dz
+\mu^\gamma f(x), \notag
\end{align}
for
$x>0$,
$0<\gamma<1$ and
$ \mu \in \mathbb{R}$ (see
\cite{KIL2}, equation (1.4)).

In the following and  analogously
to the construction of the fBm,
we will define a process by
replacing the classical
Riemann-Liouville operators with
the Hadamard ones (given in (\ref{had})
and (\ref{dermin})); to this aim we will
need the following preliminary results.

\begin{lem}\label{lem2.1}
Let
$^{H}\mathcal{D}_{-}^{ { \frac{1-\alpha}{2} }  }$
be the right-sided Hadamard derivative
defined in $(\ref{dermin}),$ then  for
$x\in \mathbb{R}_{+}$ and
$0 \leq a<b,$
\begin{equation}
\Big  ( ^{H}\mathcal{D}_{-}^{ { \frac{1-\alpha}{2} }  }
1_{[a,b)}
\Big  ) (x)=\frac{1}{\Gamma
((\alpha+1)/2 )}
\Big [
\Big  ( \log \frac{b}{x}
\Big  ) _{+}^{ { \frac{\alpha-1 }{2} }  }-
\Big  (
\log \frac{a}{x}
\Big  ) _{+}^{ { \frac{\alpha-1 }{2} } }
\Big ] ,
 \label{had4}
\end{equation}
where
$
\left ( x
\right ) _{+}
:=x1_{x\geq 0},$ and
$^{H}\mathcal{D}_{-}^{ { \frac{1-\alpha}{2} }
}1_{[a,b)}\in L^{2}(\mathbb{R}_{+}),$
  for
$\alpha \in (0,1).$
  Analogously,
let
$^{H}\mathcal{I}_{-}^{\alpha/2 }$
be the right-sided Hadamard integral
defined in $(\ref{intmin}),$ then
\begin{equation}
\left ( ^{H}\mathcal{I}_{-}^{ { \frac{\alpha-1 }{2} }  }
1_{[a,b)}
\right ) (x)=\frac{1}
{\Gamma ((\alpha+1)/2)}
\Big [
\Big ( \log \frac{b}{x}
\Big ) _{+}
^{ { \frac{\alpha-1 }{2} }  }-
\Big ( \log \frac{a}{x}
\Big ) _{+} ^{ { \frac{\alpha-1 }{2} }  }
\Big ] ,  \label{had6}
\end{equation}
and
$^{H}\mathcal{I}_{-}^{\alpha/2 }
1_{[a,b)}\in L^{2}(\mathbb{R}_{+}),$
 for
$\alpha \in (1,2).$
\end{lem}

\begin{proof}[\bf Proof]
We obtain formula (\ref{had4}) as follows, for
$0< x\leq a<b,$
\begin{align*}
\Big ( ^{H}\mathcal{D}_{-}^{ { \frac{1-\alpha}{2} }  }1_{[a,b)}
\Big ) (x) &= -\frac{x}{\Gamma
((\alpha+1)/2 )}\frac{d}{dx}\int_{a}^{b}
\Big ( \log \frac{z}{x}
\Big ) ^{ { \frac{\alpha-1 }{2} }
}\frac{dz}{z} \\
&=  -\frac{x}{\Gamma ((\alpha+1)/2 )}
\frac{d}{dx}\int_{\log (a/x)}^{\log
(b/x)}\omega ^{ { \frac{\alpha-1 }{2} }  }d\omega \\
&= \frac{1}{\Gamma ((\alpha+1)/2 )}
\Big [
\Big ( \log \frac{b}{x}
\Big )
^{ { \frac{\alpha-1 }{2} }  }-
\Big ( \log \frac{a}{x}
\Big  ) ^{ { \frac{\alpha-1 }{2} } }\Big ] .
\end{align*}
For
$0 \leq a<x<b,$ we have instead
\begin{align*}
&
\Big ( ^{H}\mathcal{D}_{-}^{ { \frac{1-\alpha}{2} }  }1_{[a,b)}
\Big ) (x)\\
&= -\frac{x}{\Gamma
((\alpha+1)/2 )}\frac{d}{dx}\int_{x}^{b}
\Big ( \log \frac{z}{x}
\Big ) ^{ { \frac{\alpha-1 }{2} }
}\frac{dz}{z}=\frac{1}{\Gamma ((\alpha+1)/2 )}
\Big ( \log \frac{b}{x}
\Big )
^{ { \frac{\alpha-1 }{2} }  },
\end{align*}
while for
$x \geq b>a$, both terms in (\ref{had4}) vanish. In
order to check the integrability properties of
$^{H}\mathcal{D}_{-}^{ { \frac{1-\alpha}{2} }  }1_{[a,b)},$
 we evaluate
\begin{align}
\int_{0}^{\infty }
\Big ( ^{H}\mathcal{D}_{-}^{ { \frac{1-\alpha}{2} }  }1_{[a,b)}
\Big )
^{2}(x)dx &= \frac{1}{\Gamma ((\alpha+1)/2)^{2}}\int_{a}^{b}
\Big ( \log
\frac{b }{x}
\Big ) ^{\alpha-1 }dx  \label{had5} \\
&=  \frac{b}{\Gamma ((\alpha+1)/2 )^{2}}
\int_{0}^{\log (b/a)}\omega ^{\alpha-1
}e^{-\omega }d\omega <\infty ,  \notag
\end{align}
for
$a<x<b<\infty
$ and, analogously, for the other cases$.$ In the case
$a=0,$ (\ref{had5}) gives
\begin{equation}
\int_{0}^{\infty }
\Big ( ^{H}\mathcal{D}_{-}^{ { \frac{1-\alpha}{2} }  }1_{[0,b)}
\Big )
^{2}(x)dx=\frac{b\Gamma (\alpha )}
{\Gamma ((\alpha+1)/2 )^{2}}<\infty , \label{der2}
\end{equation}
under the condition
$\alpha \in (0,1).$

Formula (\ref{had6}) is proved as follows, for
$0< x\leq a<b,$
\begin{align*}
\Big ( ^{H}\mathcal{I}_{-}^{ { \frac{\alpha-1 }{2} }  }1_{[a,b)}
\Big ) (x) & = \frac{1}{\Gamma
( { \frac{\alpha-1 }{2} }  )}\int_{a}^{b}
\Big ( \log \frac{z}{x}
\Big ) ^{\frac{\alpha-3}{2}}\frac{dz}{z}
\\
&= \frac{1}{\Gamma ((\alpha+1)/2)}
\Big [
\Big ( \log \frac{b}{x}
\Big )
^{ { \frac{\alpha-1 }{2} }  }-
\Big ( \log \frac{a}{x}
\Big ) ^{ { \frac{\alpha-1 }{2} }  }
\Big ]
\end{align*}
and analogously in the other cases. Finally,
\begin{equation}
\int_{0}^{\infty }
\Big ( ^{H}
\mathcal{I}_{-}^{ { \frac{\alpha-1 }{2} }  }1_{[0,b)}
\Big ) ^{2}(x)dx=\frac{b \Gamma(\alpha)}{\Gamma
((\alpha+1)/2 )^{2}}<\infty . \label{int2}
\end{equation}
\end{proof}

\section{Hadamard fractional Brownian motion}

Let
$\nu(\cdot)$ be the Gaussian measure on the space
$
\left ( \mathcal{S}
^{\prime }(\mathbb{R}),\mathcal{B}
\right )$, where
$\mathcal{B}$ is the
$\sigma$-algebra generated by the cylinder sets on
$\mathcal{S}
^{\prime }(\mathbb{R})$, i.e.,
 the unique probability measure such that
\begin{equation}
\int_{\mathcal{S}'(\mathbb{R})}
e^{i\langle x,\xi\rangle} d\nu(x)
=e^{-\frac{1}{2}\langle\xi,\xi \rangle},
 \qquad \xi \in \mathcal{S}(\mathbb{R}).
  \label{cf}
\end{equation}

Recall that for
$\nu(\cdot)$ and
$\xi, \theta \in \mathcal{S}(\mathbb{R})$,
the following hold (see
\cite{OBA}):
\begin{align}
 & \int_{\mathcal{S}'(\mathbb{R})}
 \langle x,\xi\rangle^{2n}d\nu(x)
 =
 \frac{(2n)!}{2^n n!}
 \langle \xi, \xi \rangle^n,
 \quad
 \int_{\mathcal{S}'(\mathbb{R})}
 \langle x,\xi\rangle^{2n+1}d\nu(x)
 =
 0, \label{meanodd} \\
    &  \int_{\mathcal{S}'(\mathbb{R})}
 \langle x,\xi\rangle
 \langle x, \theta \rangle d\nu(x)
 =
 \langle  \xi, \theta \rangle.  \label{autocov}
\end{align}
Thus, for any
$\phi \in \mathcal{S}(\mathbb{R})$ and
$\omega \in \mathcal{S}'(\mathbb{R})$,
we  define the random variable
$X(\phi, \omega)
:=
 \langle\omega, \phi \rangle$,
 which will be denoted, for brevity as
$X(\phi)$. As a consequence of
\eqref{cf} and
\eqref{meanodd}, the following hold, for any
$\phi, \xi \in \mathcal{S}(\mathbb{R})$
and
$k \in \mathbb{R}$,
\begin{equation}
    \mathbb{E}e^{ikX(\phi)}
 =
 e^{-\frac{k^2}{2}
 \left\Vert \phi \right \Vert^2}, \label{cf2}
    \end{equation}
\begin{equation}
\mathbb{E}e^{ik\left[X(\phi)-X(\xi)\right]}
 =
 e^{-\frac{k^2}{2}\left\Vert\phi-\xi\right\Vert},
   \label{cf3}
\end{equation}
\begin{equation}\mathbb{E}
\left[ X(\phi)^2
\right]
 =
 \left\Vert \phi \right \Vert^2, \label{va}
\end{equation}
where
$\left\Vert \cdot \right\Vert^2
:=
 \langle \cdot,\cdot \rangle.$
It follows from \eqref{va} that the definition of
$X(\cdot)$ can be easily extended to any function in
$L^2(\mathbb{R})$ (see, for example
\cite{BOC}). Thus, by considering
Lemma \ref{lem2.1},
we are able to give the following definition.

\begin{definition}\label{def3.1}
Let
$^{H}\mathcal{D}_{-}^{ { \frac{1-\alpha}{2} }  }$ and
$^{H}\mathcal{I}_{-}^{ { \frac{\alpha-1 }{2} }  }$ be
the right-sided Hadamard derivative and
integral defined in $(\ref{dermin})$
and $(\ref{intmin}),$ respectively.
Then we define, on the probability space
$
\left ( \mathcal{S}
^{\prime }(\mathbb{R}),\mathcal{B}, \nu
\right )
$, the \textit{Hadamard-fractional
Brownian motion}
$($hereafter Hadamard-fBm$)$
$
B^H_{\alpha }
:=\left\{ B^H_{\alpha
}(t)\right\} _{t\geq 0}$ as
\begin{equation}
B_{\alpha }^H (t,\omega )
:= \big
 \langle \omega , \prescript{H}{}{\mathcal{M
}}_{-}^{\alpha/2 }1_{[0,t)}\big \rangle ,
\qquad t\geq 0,\;\omega \in \mathcal{S
}^{\prime }(\mathbb{R}), \label{Hfbm}
\end{equation}
where
\begin{equation}
\big ( ^{H}\mathcal{M}_{-}^{\alpha/2 }f
\big ) (x)
:=\left\{
\begin{array}{l}
K_{\alpha}
\big ( ^{H}\mathcal{D}_{-}^{ { \frac{1-\alpha}{2} }  }f
\big ) (x),\qquad \alpha \in (0,1)
\\
f(x),\qquad \alpha
 =
 1 \\
K_{\alpha}
\big ( ^{H}\mathcal{I}_{-}^{ { \frac{\alpha-1 }{2} }  }f
\big ) (x),\qquad \alpha \in (1,2)
\end{array}
\right. \label{ma}
\end{equation}
for
$K_{\alpha}
 =
 \Gamma ((\alpha+1)/2)/\sqrt{\Gamma(\alpha)}$.
\end{definition}

In view of what follows,
we recall the Tricomi's confluent
hypergeometric function (see
\cite{NIST}, formula (13.2.42)) defined as
\[
\Psi
\left (a, b ; z
\right )
:=\frac{\Gamma(1-b)}{\Gamma(1+a-b)}
\Phi(a,b;z)+\frac{\Gamma(b-1)}
{\Gamma(a)}\Phi(1+a-b,2-b;z),
\]
for
$a,b,z \in \mathbb{C}$,
$\Re(b) \neq 0, \pm 1, \pm 2,...$,
where
$$
\Phi(a,b;z)
:=\sum_{l=0}^{\infty}
\frac{(a)_l}{(b)_l}\frac{z^l}{l!}
 \ \text{ and }
 (c)_l
:=\frac{\Gamma(c+l)}{\Gamma(c)} .
$$
In what follows, we will restrict to the case
$\Psi
\left (a, b ; z
\right )$, for
$a,b,z \in \mathbb{R}$.

We recall that the following
asymptotic behaviors hold, as
$z \rightarrow 0$ (see
\cite{NIST}, formulae (13.2.22), (13.2.20),
and (13.2.18), respectively):
\begin{equation}
\Psi(a,b;z)
 =
 \frac{\Gamma(1-b)}{\Gamma(a-b+1)}+O(z),
 \qquad \mathcal{R}(b) <0, \label{asy0}
\end{equation}
\begin{equation}
\Psi(a,b;z)
 =
 \frac{\Gamma(1-b)}{\Gamma(a-b+1)}+
 O(z^{1-\mathcal{R}(b)}), \qquad
 \mathcal{R}(b) \in (0,1), \label{asy1}
\end{equation}
\begin{equation}
\Psi(a,b;z)
 =
 \frac{\Gamma(1-b)}
{\Gamma(a-b+1)}+\frac{\Gamma(b-1)}
{\Gamma(a)}z^{1-b}+O(z^{2-\mathcal{R}(b)}),
 \qquad \mathcal{R}(b) \in (1,2). \label{asy2}
\end{equation}

In what follows, we use the following formula:
\begin{equation}
\Psi(a,b;z)
 =
 z^{1-b}\Psi(a+1-b,2-b;z), \label{psi}
\end{equation}
(see
\cite{NIST}, eq.\ (13.2.40)).
Moreover, the following integral
representation holds for the confluent
hypergeometric function:
\begin{equation} \label{eqTricomiINtegral} \Psi
\left (a, b ; z
\right )
 =
 \frac{1}{\Gamma(a)}
 \int^{\infty}_{0}e^{-sz}s^{a-1}(1+s)^{b-a-1}ds,
  \end{equation} if
$\Re(a)>0,$
$\Re(z) \geq 0$ (see
\cite{KIL}, p.\ 30).
It is easy to check that the function
$\Psi(a,b;\cdot)$ is non-increasing
(resp.\ non-decreasing) for
$a>0$ (resp.\
$a<0$), on
$\mathbb{R}^+$, by taking into account
\begin{equation}
\frac{d}{dx}\Psi(a,b;x)
 =
 -a\Psi(a+1,b+1;x), \label{derpsi}
\end{equation}
(see
\cite{NIST}, formula (13.3.22)),
together with \eqref{eqTricomiINtegral}
(and \eqref{psi}, for
$a <0$).

\begin{theorem}\label{proc}
    For any
$\alpha \in (0,1)  \cup (1,2)$,
the Hadamard-fBm is a Gaussian process,
with zero mean,
 \begin{equation}
var(B^H_{\alpha }(t))
 =
 t, \qquad  t\geq 0,\label{var}
\end{equation}
    and
\begin{equation}
cov(B^H_{\alpha }(t),B^H_{\alpha }(s))
=C_\alpha (s \wedge t)\Psi
\Big (\frac{1-\alpha}{2}, 1-\alpha ;
 \log
\Big (\frac{s \vee t}
 { s \wedge t}
\Big )
\Big ), \label{cov}
\end{equation}
$ s,t \in \mathbb{R}_+, $
 $ s \neq t, $
where
$C_\alpha
 =
 2^{1-\alpha}\sqrt{\pi}/\Gamma(\alpha/2)$.
Moreover, its characteristic function reads,
 for
$0<t_1<...<t_n$,
$n \in \mathbb{N}$ and
$k_j \in \mathbb{R},$
$j=1,...,n$,
\begin{equation}
    \mathbb{E}e^{i\sum_{j=1}^{n}k_j
    B^H_{\alpha }(t_j)}=
    \exp
\Big \{-\frac{1}{2}
    \sum_{j,l=1}^{n}k_j k_l
     \sigma^\alpha_{j,l}
\Big  \}, \label{cf4}
\end{equation}
where
\begin{equation}
\sigma^\alpha_{j,l}
:=\left\{
\begin{array}{l}
t_j , \qquad j=l \\
C_\alpha (t_j \wedge t_l)\Psi

\left (\frac{1-\alpha}{2}, 1-\alpha ;
 \log
\left (\frac{t_j \vee t_l}
 { t_j \wedge t_l}
\right )
\right ) , \qquad j \neq l.
 \label{sigma}
\end{array}
\right.
\end{equation}
\end{theorem}

\begin{proof}[\bf Proof]
Gaussianity  follows by the linearity of Def.\
\ref{def3.1} and
$\mathbb{E}(B^H_{\alpha }(t))
 =
 0$ by taking into account
equation (\ref{meanodd}). The variance can be obtained by
considering (\ref{meanodd}) together with equations
(\ref{der2}) and (\ref{int2}), respectively. For the
autocovariance and for
$\alpha \in (0,1)$, we consider Lemma \ref{lem2.1} and the
following
$L^{2}$-inner product
\begin{align}
& \big
 \langle ^{H}\mathcal{D}_{-}^{ { \frac{1-\alpha}{2} }
}1_{[0,s)},^{H}\mathcal{D}
_{-}^{ { \frac{1-\alpha}{2} }
}1_{[0,t)}\big \rangle \notag \\
&= \int_{\mathbb{R}_+}
\Big ( ^{H}
\mathcal{D}_{-}^{ { \frac{1-\alpha}{2} }  }1_{[0,s)}
\Big )
(x)
\Big (
^{H}\mathcal{D}
_{-}^{ { \frac{1-\alpha}{2} }  }1_{[0,t)}
\Big ) (x)dx \label{inn} \\
&=\frac{1}{\Gamma^2 ((1+\alpha)/2 )}\int_{0}^{\infty}
\Big ( \log \frac{s}{x}
\Big )_+
^{ { \frac{\alpha-1 }{2} }  }
\Big ( \log \frac{t}{x}
\Big )_+
^{ { \frac{\alpha-1 }{2} }  }dx \notag \\
&\overset{\text{for
$s<t$}}{=}\frac{1}{\Gamma^2 ((1+\alpha)/2 )}\int_{0}^{s}
\Big ( \log \frac{s}{x}
\Big )
^{ { \frac{\alpha-1 }{2} }  }
\Big ( \log \frac{t}{x}
\Big )
^{ { \frac{\alpha-1 }{2} }  }dx  \notag\\
&=\frac{s}{\Gamma^2 ((1+\alpha)/2 )}\int_{0}^{\infty}
\Big [ \log
\Big (\frac{t}{s}
\Big )+w\Big ]
^{ { \frac{\alpha-1 }{2} }  }w^{ { \frac{\alpha-1 }{2} } }e^{-w} dw \notag \\
&=\frac{s}{\Gamma^2 ((1+\alpha)/2 )}\Big[ \log
\frac{t}{s}\Big]^{\alpha
}\int_{0}^{\infty}
(1+y)^{ { \frac{\alpha-1 }{2} } }y^{ { \frac{\alpha-1 }{2} } }e^{-y
\log(t/s)} dy \notag \\
&=\frac{s}{\Gamma ((1+\alpha)/2 )}\Big[ \log
\frac{t}{s}\Big]^{\alpha
}\Psi
\Big (\frac{\alpha+1}{2}, \alpha+1 ;
\log
\Big (\frac{t}{ s}
\Big )
\Big ) , \notag
\end{align}
by taking into account \eqref{eqTricomiINtegral}.
We now apply formula \eqref{psi}
for
$a=(\alpha+1)/2$,
$b=\alpha+1$ and
$z=\log(t/s)$. In order to get (\ref{cov}), in the case
$s<t$, we must consider the constant
$K_\alpha$ given in (\ref{ma}) together with the
duplication formula of the gamma function.  The cases
$s \geq t$ and
$\alpha \in (1,2)$ follow analogously.

By considering \eqref{cf2} together with \eqref{var} and
\eqref{cov}, we get
\begin{equation}
    \mathbb{E}e^{i\sum_{j
 =
 1}^{n}k_j B^H_{\alpha }(t_j)}
 =
    \exp \Big\{-\frac{1}{2}\Big\Vert \sum_{j=1}^{n}k_j
\mathcal{M}^{\alpha/2}_{-}1_{[0,t_j)}\Big\Vert^2\Big\}.
 \label{cf5}
\end{equation}
Formula \eqref{cf4} with  \eqref{sigma} follows by taking
into account that,
for any
$\alpha \in (0,1) \cup (1,2)$,
    \begin{equation}
    \lim_{x \to 1^+} C_\alpha\Psi
\Big (\frac{1-\alpha}{2},
1-\alpha ; \log x
\Big )=1, \label{lim}
    \end{equation}
    recalling \eqref{asy1}, for
$\alpha \in (0,1)$, and \eqref{asy0}, for
$\alpha \in (1,2),
$
together with the duplication formula of the gamma
function.
\end{proof}

\begin{coro} \label{corostat}
    The H-fBm is self-similar with index
$1/2$ and has non-stationary increments, with
characteristic function
    \begin{align}
    &
        \mathbb{E}e^{ik
\left (B^H_{\alpha
}(t)-B^H_{\alpha
}(s)
\right )} \notag
\\
&
= \exp
\Big \{-\frac{k^2}{2}\Big [t+s
-2C_\alpha   (t \wedge
s)\Psi
\Big  (\frac{1-\alpha}{2},
1-\alpha ; \log
\Big  (\frac{t \vee s}{ t \wedge
s}
\Big  )
\Big  )  \Big ]\Big \}, \label{cf6}
    \end{align}
    for
$k \in \mathbb{R}$ and
$s,t \geq 0$. The increments' covariance on
non-overlapping
intervals is negative $($resp. positive$),$ for
$\alpha \in (0,1)$ $($resp.\
$\alpha \in (1,2) ).$
\end{coro}

\begin{proof}[\bf Proof]
By considering \eqref{cf4}, we can write, for any
$a \in \mathbb{R}$, that
   \begin{align*}
 &  \mathbb{E}e^{i\sum_{j=1}^{n}k_j B^H_{\alpha }(at_j)}
\notag\\
    &=  \exp\!
\Big \{ \! -\frac{a}{2}\Big [\sum_{j=1}^{n}k^2_j
t_j +C_\alpha \sum_{j \neq l} k_j k_l (t_j \wedge
t_l)\Psi
\Big  (\frac{1-\alpha}{2},
1-\alpha ;
\log
\Big  (\frac{t_j \vee t_l}{ t_j \wedge t_l}
\Big  )
\Big  )  \Big ]\Big \} \notag \\
&= \mathbb{E}e^{i{a^{1/2}\sum_{j=1}^{n}k_j
B^H_{\alpha }(t_j)}} \notag
    \end{align*}
    so that
$$
\left\{B^H_{\alpha }(at) \right\}_{t \geq 0}
\overset{f.d.d.}{=} \left\{a^{1/2} B^H_{\alpha }(t)
\right\}_{t \geq 0} ,
$$
 where
$ \overset{f.d.d.}{=}$ denotes equality of the
finite-dimensional distributions. Formula \eqref{cf6} is
obtained by \eqref{cf4}, for
$k_1=k$ and
$k_2=-k$, and the non-stationarity of the increments
easily follows.

    Let
$0<u<v<s<t$ and let
$$
C(u,v;s,t)
:=\mathbb{E}\left[(B^H_{\alpha
}(t)-B^H_{\alpha
}(s))(B^H_{\alpha }(v)-B^H_{\alpha }(u))
\right] ,
$$
 then it is easy to check that
   \begin{align*}
   C(u,v;s,t)
   & = vC_\alpha
\Big [
\Psi
\Big  (\frac{1-\alpha}{2},
1-\alpha ; \log
\Big  (\frac{t}{v}
\Big  )
\Big  )-\Psi
\Big  (\frac{1-\alpha}{2},
1-\alpha ;
\log
\Big  (\frac{s}{v}
\Big  )
\Big  )\Big ]    \notag\\
   & +  u C_\alpha
\Big [ \Psi
\Big  (\frac{1-\alpha}{2},
1-\alpha
; \log
\Big  (\frac{t}{u}
\Big  )
\Big  )-\Psi
\Big  (\frac{1-\alpha}{2},
1-\alpha ;
\log
\Big  (\frac{s}{u}
\Big  )
\Big  )\Big ]. \notag
   \end{align*}
   Thus,
$C(u,v;s,t)<0$ (resp.
$<0$), for
$\alpha \in (0,1)$ (resp.\
$\alpha \in (1,2)$), by considering that
$\Psi(a,b;\cdot)$ is non-increasing (resp.\
non-decreasing)
for
$a>0$ (resp. $a<0$).
\end{proof}

\begin{remark}\rm
    We note that the variance of $B_{\alpha}^H$ is linear
in $t$, for any $\alpha \in (0,2)$, so that the effect of
the Hadamard operator vanishes on the one-dimensional
distribution, and the process displays a diffusing
behavior as the standard Brownian motion. It can be
checked that the same result would be obtained for any
fractional operator whose
Mellin transform is equal to the Mellin transform of the
indicator function multiplied by a quantity depending on
$\alpha$; thus it would also hold for
any Erd\'{e}lyi-Kober type operator (see
\cite{KIL}, sec.\ 2.6).

On the other hand the H-fBm shares with the fBm the sign
of increments' covariance over non-overlapping intervals.
 \end{remark}

   \begin{theorem}
       Let $\left\{ B (t)\right\}_{t\geq 0}$ be the
standard Brownian motion, then the following relationships
hold, for $\alpha \in (0,1)$ $($resp.\ $\alpha \in (1,2) ),$
\begin{equation}       P
\Big  (\sup_{0 \leq s \leq t} B^H
_\alpha (s)>x
\Big  ) \underset{(\text{resp.\ }
\leq)}{\geq}
 P
\Big  (\sup_{0 \leq s\leq t} B (s)>x
\Big  )
\label{geq}
\end{equation}
\begin{equation}
\mathbb{E}
\Big  (\sup_{0 \leq s \leq t} B^H _\alpha (s)
\Big  ) \underset{(\text{resp.\ } \leq)}{\geq}
\mathbb{E}
\Big  (\sup_{0
\leq s\leq t} B (s)
\Big  ), \label{geq2}
\end{equation}
for any $t \geq 0$ and $x \in \mathbb{R}$.
Moreover, $\left\{ B_\alpha ^H (t)\right\}_{t\geq 0}$ is
stochastically continuous, for any $\alpha \in (0,1) \cup
(1,2)$, and has continuous sample paths a.s., for $\alpha
\in (1,2)$.
   \end{theorem}

\begin{proof}[\bf Proof]
       As a consequence of \eqref{derpsi},
 for any $x>1$, the function $\Psi
  (\frac{1-\alpha}{2}, 1-\alpha ;$ $ \log x
  )$ is non-increasing, for
$\alpha \in (0,1)$, and non-decreasing, for
$\alpha \in (1,2)$. Therefore,
taking into account \eqref{lim}, we get that
    \begin{equation}
    0 \leq C_\alpha\Psi
\Big (\frac{1-\alpha}{2}, 1-\alpha ; \log x
\Big )\leq  1 \leq \sqrt{x}, \qquad
\alpha \in (0,1), x>1, \label{ub1}
   \end{equation}
while it can be proved that
    \begin{equation}
   1 \leq C_\alpha\Psi
\Big (\frac{1-\alpha}{2}, 1-\alpha ; \log x
\Big )\leq  \sqrt{x},
\qquad \alpha \in (1,2), x>1 .\label{ub1bis}
   \end{equation}
   Indeed, we define
   $$
   h(x)
:=\sqrt{x}- C_\alpha\Psi
\Big (\frac{1-\alpha}{2}, 1-\alpha ; \log x
\Big )
$$
and check that $h(x)>h(1)=0$,
for any $x>1:$ $h(1)=0,$ by \eqref{lim},
and the inequality follows by taking
the first derivative
   \begin{align}
   h'(x) & = \frac{1}{2\sqrt{x}}-
   \frac{\alpha -1}{2x} C_\alpha \Psi
\Big (\frac{3-\alpha}{2}, 2-\alpha ; \log x
\Big ) \notag \\
   & =  \frac{1}{2\sqrt{x}}
\Big[1-\frac{\alpha -1}{\sqrt{x}}
   C_\alpha \Psi
\Big (\frac{3-\alpha}{2}, 2-\alpha ; \log x
\Big )\Big]>0,
   \end{align}
   for any $x>1$, since $\Psi
\left (\frac{3-\alpha}{2},
2-\alpha ; \log \cdot
\right )$ is non-increasing, for
$\alpha \in (1,2)$ (by \eqref{derpsi}),
 and
   \begin{equation}
   \lim_{x \to 1^+}(\alpha -1)C_\alpha \Psi
\Big (\frac{3-\alpha}{2}, 2-\alpha ; \log x
\Big )=\frac{(\alpha -1)\Gamma(\alpha -1)}
{\Gamma ((1+\alpha)/2)}\frac{2^{1-\alpha}
\sqrt{\pi}}{\Gamma(\alpha/2)}=1, \label{lim1}
   \end{equation}
   by considering again \eqref{asy1}
   and the duplication formula.

   Therefore, we have that, for any
$s,t \geq 0$,
   \begin{equation}
 cov(B^H_{\alpha }(t),
   B^H_{\alpha }(s)) \underset{(\text{resp.\ }
   \geq)}{\leq}  s \wedge t= cov(B(t),B(s)),
    \label{coco}
    \end{equation}
    for $\alpha \in (0,1) $ (resp.
    $\alpha \in (1,2)$), by \eqref{ub1}
    (resp. \eqref{ub1bis}).

   Formula \eqref{geq} follows from
   \eqref{coco} by considering that
$\left\{ B_\alpha ^H (t)\right\}_{t\geq 0}$
and
$\left\{ B (t)\right\}_{t\geq 0}$
are Gaussian, centered, with the same
variance and by applying the
Slepian inequality,
 on the separable space $[0,t]$
 (see, for details,
\cite{ADL}, Corollary 2.4).
Analogously, we obtain \eqref{geq2}
by the Sudakov-Fernique inequality (see
\cite{ADL}, Theorem 2.9, for details).

The stochastic continuity of
$\left\{ B_\alpha ^H (t)\right\}_{t\geq 0}$
follows by proving the continuity of
its incremental variance, since then,
for any $t \geq 0$ and $\varepsilon >0$,
\[
\lim_{h \to 0}P\left\{|B_\alpha ^H (t+h)-
B_\alpha ^H (t)|> \varepsilon \right\}
\leq
\lim_{h \to 0}
\frac{\mathbb{E}\left[B_\alpha ^H (t+h)
-B_\alpha ^H (t)\right]^2}{\varepsilon^2}= 0.
\]
By considering \eqref{inn} and denoting,
for brevity $g_t (x)
:= \log \frac{t}{x}$ and $g_t(x)_+
:=g_t(x)1_{g_t(x)\geq0}=g_t(x)1_{0<x
\leq t}$, we can write, for $0\leq s<t$,
\begin{align}
& \rho_H (s,t)
:=  \mathbb{E}
\Big [ B_\alpha ^H (t)-B_\alpha ^H (s)
\Big ]^2 \label{rho} \\
=  & \frac{1}{\Gamma(\alpha)}
\Big \{ \int_{0}^{t}  \! g_t (x)
^{\alpha-1 }dx+\int_{0}^{s}  \! g_s (x)
^{\alpha-1 } dx-2 \int_{0}^{s} \! g_t (x)
^{ { \frac{\alpha-1 }{2} }  } g_s (x)
^{ { \frac{\alpha-1 }{2} }  } dx\Big \} \notag \\
= {} &\frac{1}{\Gamma(\alpha)}
\Big \{\int_{s}^{t}g_t (x)
^{\alpha-1 }dx + \int_{0}^{s}
\Big [g_t (x)
^{ { \frac{\alpha-1 }{2} }  }-g_s (x)
^{ { \frac{\alpha-1 }{2} }  }\Big ]^2 dx\Big \} \notag \\
= {} &\frac{1}{\Gamma(\alpha)}
\Big \{\int_{s}^{t} \Big [g_t (x)
^{ { \frac{\alpha-1 }{2} }  }_+ -g_s (x)
^{ { \frac{\alpha-1 }{2} }  }_+ \Big ]^2 dx
  +    \int_{0}^{s} \Big [g_t (x)_+
^{ { \frac{\alpha-1 }{2} }  }-g_s (x)_+
^{ { \frac{\alpha-1 }{2} }  }\Big ]^2 dx\Big \}
 \notag \\
= {} &\frac{1}{\Gamma(\alpha)}
\int_{0}^{t} \Big [
\Big  ( \log \frac{t}{x}
\Big  )_+
^{ { \frac{\alpha-1 }{2} }  }-
\Big  ( \log \frac{s}{x}
\Big  )_+
^{ { \frac{\alpha-1 }{2} }  }\Big ]^2 dx . \notag
\end{align}
It is evident by \eqref{rho} that
$\lim_{h \to 0} \rho_H (t,t+h)=0$,
for any $t \geq 0$ and for
$\alpha \in (0,1) \cup (1,2)$.

In the case
$\alpha \in (1,2)$, in order to
prove the a.s.\ continuity of the
 trajectories,  we proceed as follows:
 we first apply Theorem 6.2 in
\cite{PAN},
 in order to check that the Gaussian process
 $\left\{B_\alpha^H(t)\right\}_{t\geq 0}$
 can be viewed as a random element in the space
 $\mathbb{D}$
 (of the real valued cádlág
 functions on $\mathbb{R}_+$)
 with the specified
 finite-dimensional distributions.
 This is verified since the sufficient
 condition holds: i.e.,
 for any $0 \leq r<s<t$ and $C>0$,
\begin{align*}
&\mathbb{E}\left[ |B_\alpha ^H (r)
-B_\alpha ^H (s)|^2
|B_\alpha ^H (s)-B_\alpha ^H (t)|^2\right] \\
&\leq \left\{\mathbb{E}
\left[ |B_\alpha ^H (r)-B_\alpha ^H (s)|\right]^4
\mathbb{E}\left[ |B_\alpha ^H (s)-B_\alpha ^H (t)|
\right]^4 \right\}^{1/2} \notag \\
&=  C \left\{\rho_H (r,s)^2 \rho_H (s,t)^2
\right\}^{1/2} \leq C \rho_H (r,t)^2, \notag
\end{align*}
by the H\"{o}lder inequality and the
properties of the Gaussian moments
 (for $C=3$).
The last inequality follows by proving that
$$
\rho_{H}(r,s)+\rho_{H}(s,t) \leq \rho_{H}(r,t) ,
\ \text{ for any $0 \leq r<s<t$.}
$$
 Indeed,
for any $0 \leq r<s<t$, we get
\begin{align}
& \rho_{H}(r,t)
 -\rho_{H}(r,s)-\rho_{H}(s,t) \notag \\
= {}
& \frac{1}{\Gamma(\alpha)}
\Big \{\int_{0}^{t}
\Big [g_t (x)
^{ { \frac{\alpha-1 }{2} }  }_+ -g_r (x)
^{ { \frac{\alpha-1 }{2} }  }_+ \Big ]^2 dx-\int_{0}^{s}
 \Big [g_s (x)
^{ { \frac{\alpha-1 }{2} }  }_+ -g_r (x)
^{ { \frac{\alpha-1 }{2} }  }_+ \Big ]^2 dx  \notag \\
&-   \int_{0}^{t} \Big [g_t (x)
^{ { \frac{\alpha-1 }{2} }  }_+ -g_s (x)
^{ { \frac{\alpha-1 }{2} }  }_+ \Big ]^2 dx   \Big \} \notag \\
= {} & \frac{1}{\Gamma(\alpha)}\Big \{\int_{0}^{t}
\Big [g_t (x)_+
^{ { \frac{\alpha-1 }{2} }  } -g_s (x)_+
^{ { \frac{\alpha-1 }{2} }  } +g_s (x)_+
^{ { \frac{\alpha-1 }{2} }  }  -g_r (x)_+
^{ { \frac{\alpha-1 }{2} }  } \Big ]^2 dx   \notag \\
&-  \int_{0}^{t}
\Big [g_s (x)_+
^{ { \frac{\alpha-1 }{2} }  }  -g_r (x)_+
^{ { \frac{\alpha-1 }{2} }  } \Big ]^2 dx+\int_{s}^{t}
 \Big [g_s (x)_+
^{ { \frac{\alpha-1 }{2} }  }  -g_r (x)_+
^{ { \frac{\alpha-1 }{2} }  } \Big ]^2 dx    \notag \\
&-   \int_{0}^{t} \Big [g_t (x)
^{ { \frac{\alpha-1 }{2} }  }  -g_s (x)
^{ { \frac{\alpha-1 }{2} }  } \Big ]^2 dx \Big \} \notag \\
= {} & \frac{1}{\Gamma(\alpha)}
\Big \{2\int_{0}^{t}
\Big [g_t (x)_+
^{ { \frac{\alpha-1 }{2} }  } -g_s (x)_+
^{ { \frac{\alpha-1 }{2} }  }\Big ] \Big [g_s (x)_+
^{ { \frac{\alpha-1 }{2} }  }  -g_r (x)_+
^{ { \frac{\alpha-1 }{2} }  } \Big ] dx   \notag \\
&+   \int_{s}^{t} \Big [g_s (x)
^{ { \frac{\alpha-1 }{2} }  }  -g_r (x)
^{ { \frac{\alpha-1 }{2} }  } \Big ]^2 dx \Big \}
 \geq 0,\notag
\end{align}
where the last inequality holds only for
$\alpha >1$.

 Finally, we apply Theorem 1 in
\cite{HAH},
which states that, if a real-valued
Gaussian process with sample
paths in $\mathbb{D}$ is stochastically
continuous (or, equivalently, in quadratic mean),
 then it has continuous sample
paths almost surely.
   \end{proof}

\begin{remark}\rm
    Another consequence of
    \eqref{ub1} and
    \eqref{ub1bis} is that the
    Cauchy-Schwartz inequality is
    satisfied by \eqref{var} and
\eqref{cov}.
\end{remark}

\begin{remark}\rm
 The sample paths' continuity of the H-fBm, for
$\alpha \in (1,2)$, can be alternatively
derived from the same property holding for
the Brownian motion, by applying Lemma 3.2 in
\cite{MAR} and considering that, by
\eqref{cov},
the incremental variance of
$\left\{ B_\alpha ^H (t)\right\}_{t\geq 0}$
can be bounded by that of
$\left\{ B (t)\right\}_{t\geq 0}$, i.e.,
$$
 \rho_H(s,t) \leq
\mathbb{E}\left[ B(t)-B (s)\right]^2 ,
\ \text{
for any $s,t \geq 0.$}
$$
\end{remark}

In order to analyse the long-time
properties of the Hadamard-fBm,
let us define the discrete-time
increment process
$X_{\alpha}^H
:=\left\{ X^H_{\alpha
}(n)\right\} _{n\geq 1}$,
where
$X^H_{\alpha}(n)
:=B_{\alpha}^H(n)-B_{\alpha}^H(n-1)$,
for
$n \in \mathbb{N}.$ Since, by
Corollary \ref{corostat}, the increments of
$B_{\alpha}^H$ are non-stationary,
we apply the criterion for long/short
range dependence given in
\cite{HEY} for this kind of processes,
that is we study the asymptotic behavior of
\begin{equation}\label{delta}
\Delta_t ^{(m)}
:=\frac{var\left[ \sum_{j=tm-m+1}^{tm}
X_\alpha ^H (j)\right]}{\sum_{j=tm-m+1}^{tm}var
\left[X_\alpha ^H (j)\right]} \qquad t,m
 \in \mathbb{N}.
\end{equation}
In particular, analogously to the case of
stationary increments, we will say that the process

\begin{itemize}
\item  is antipersistent if $\Delta_t ^{(m)} \to 0$
\item has short memory if $\Delta_t ^{(m)} \to K>0$
\item has long memory if $\Delta_t ^{(m)} \to \infty$

\end{itemize}
as $m \to\infty$.

\begin{theorem}\label{cor2.1}
 The discrete-time increment process
 $\left\{X^H_{\alpha}(n)\right\}_{n\geq 1} $
 is anti-persistent for
 $\alpha \in (0,1)$, while it is
 long-range dependent for
 $\alpha \in (1,2)$.
\end{theorem}

\begin{proof}[\bf Proof]
It is easy to see that, for any
$\alpha \in (0,1) \cup (1,2)$,
the numerator of $\Delta_t ^{(m)}$
reduces to
\begin{align}\label{lrd}
& var\Big [ \sum_{j=tm-m+1}^{tm}
X_\alpha ^H (j)\Big ]
 = var\left[ B_\alpha ^H (tm)-
B_\alpha ^H (tm-m)\right]  \\
&= m(2t-1)-2m(t-1)C_\alpha \Psi
\Big  (\frac{1-\alpha}{2}, 1-\alpha ; \log
\Big  (\frac{t}{ t-1 }
\Big  )
\Big  )= m C_{\alpha,t} , \notag
\end{align}
where $C_{\alpha,t}$ is a positive
constant not depending on $m$. Indeed,
 by \eqref{ub1} and \eqref{ub1bis},
  we have that
\begin{align*}
C_{\alpha,t} & =
2t-1-2(t-1)C_\alpha \Psi
\Big (\frac{1-\alpha}{2}, 1-\alpha ; \log
\Big (\frac{t}{ t-1 }
\Big )
\Big ) \\
&
 \geq 2t-1-2\sqrt{t(t-1)} \geq 0, \ \text{for $t >1$.}
\end{align*}

As far as the denominator is concerned,
in the case $\alpha \in (0,1)$,
we have that
\begin{align}
&\sum_{j=tm-m+1}^{tm}var
\left[X_\alpha ^H (j)\right] \label{lrd2}\\
&=\sum_{j=tm-m+1}^{tm} (2j-1)-2
C_{\alpha}\sum_{j=tm-m+1}^{tm}(j-1)\Psi
\Big (\frac{1-\alpha}{2}, 1-\alpha ; \log
\Big (\frac{j}{ j-1 }
\Big )
\Big )   \notag \\
&=:m^2 (2t-1)-2C_\alpha S_{t,m}. \notag
\end{align}
Since the function
$\Psi(a,b;\cdot)$ is non-increasing   for
$a>0$, and thus
$\Psi(a,b;$ $\log(x/(x-1)))$ is
non-decreasing, for $x>1$, the term
$S_{t,m}$ can be bounded as follows:
\begin{align}
S_{t,m}  & = \sum_{j=tm-m+1}^{tm}(j-1)\Psi
\Big (\frac{1-\alpha}{2}, 1-\alpha  ;
 \log
\Big (\frac{j}{ j-1 }
\Big )
\Big ) \label{lrd3} \\
&\leq   \frac{1}{2}\left[m^2
  (2t-1
  )-m\right] \Psi
\Big (\frac{1-\alpha}{2}, 1-\alpha  ; \log
\Big (\frac{tm}{ tm-1 }
\Big )
\Big ).\notag
\end{align}
Therefore, we have that
\[
\sum_{j=tm-m+1}^{tm}var \left[X_\alpha ^H (j)\right]
 \geq
m^2 (2t-1)-\left[m^2
\left (2t-1
\right )-m\right] K^{(m)}_{\alpha ,t}
\]
where
$$
K^{(m)}_{\alpha,t} =C_\alpha \Psi
\Big (\frac{1-\alpha}{2}, 1-\alpha ; \log
  (  {\frac{tm}{tm-1}}
  )
\Big ) .
$$
By taking into account
\eqref{lrd} and  applying the
l'H\^{o}pital rule together with
\eqref{derpsi}, it can be proved that,
for any $t>1$,
\begin{align*}
& \liminf_{m \to \infty}\frac{1}{\Delta_t^{(m)}}
\geq \frac{1}{C_{\alpha, t}}
\lim_{m \to \infty} \left[m (2t-1)
-\left [m
\left (2t-1
\right )-1\right] K^{(m)}_{\alpha ,t}\right]
 \notag \\
& = \frac{1}{C_{\alpha, t}}\lim_{m \to \infty}
 \frac{2t-1-\left[2t-1-\frac{1}{m}\right]
 C_\alpha \Psi
\big (\frac{1-\alpha}{2}, 1-\alpha ; \log
\big (  {\frac{tm}{tm-1}}
\big )
\big )}{1/m}
 \notag \\
& =
    \frac{1}{C_{\alpha, t}} \lim_{m \to \infty}
\frac{-\frac{1}{m^2}K^{(m)}_{\alpha ,t} \! + \!
 [2t-1-\frac{1}{m} ]\frac{\alpha -1}
{2m(tm-1)}C_\alpha \Psi
\big (\frac{3-\alpha}{2}, 2-\alpha ; \log
\big (  {\frac{tm}{tm-1}}
\big )
\big )}{-1/m^2}
 \notag \\
& = \frac{1}{C_{\alpha, t}} \lim_{m \to \infty}
K^{(m)}_{\alpha}  \\
&
\quad
+
\Big [2t-1-\frac{1}{m}
\Big ]\frac{(1-\alpha)m }{2(tm-1)}C_\alpha \Psi
\Big  (\frac{3-\alpha}{2}, 2-\alpha ; \log
\Big  (  {\frac{tm}{tm-1}}
\Big  )
\Big  )=\infty, \notag
\end{align*}
for $\alpha \in (0,1)$, where, in the last step,
 we have applied
\eqref{lim} and \eqref{asy2}.
Therefore,
$\lim_{m \to \infty} \Delta_t ^{(m)} =0$
and the anti-persistence follows.

For $\alpha \in (1,2)$, the function
$var \left[X_\alpha ^H (j)\right] $
(see \eqref{lrd2}) is non-increasing, for
$j>1$, as can be ascertained by
differentiating
\[
f(x)=var\left[X_\alpha ^H (x)\right]
 =(2x-1)-2(x-1) C_{\alpha}\Psi
\Big (\frac{1-\alpha}{2}, 1-\alpha ; \log
\Big (\frac{x}{ x-1 }
\Big )
\Big ),
\]
for $x>1.$ We have that
\begin{align*}
f'(x)=2 & -2C_{\alpha}\Psi
\Big (\frac{1-\alpha}{2}, 1-\alpha ; \log
\Big (\frac{x}{ x-1 }
\Big )
\Big ) \\
&
+\frac{\alpha -1}{x}C_{\alpha}\Psi
\Big (\frac{3-\alpha}{2}, 2-\alpha ; \log
\Big (\frac{x}{ x-1 }
\Big )
\Big ),
\end{align*}
which is negative, since
$f''(x)>0$ and
$\lim_{x \to \infty}f'(x)=0$
(taking into account \eqref{lim}
and \eqref{lim1}).
 Therefore, we have that
\begin{align*}
\sum_{j=tm-m+1}^{tm}var
\left[X_\alpha ^H (j)\right]
&
 \leq m \; var
\left[X_\alpha ^H (tm-m+1)\right] \\
&
= 2(t-1)m^2+m -2(t-1)m^2 K'^{(m)}_{\alpha ,t},
\end{align*}
where
$$
K'^{(m)}_{\alpha,t} =C_\alpha \Psi
\Big (\frac{1-\alpha}{2}, 1-\alpha ; \log
\Big (tm-m+1 /(tm-m)
\Big )
\Big ) .
$$
By applying the l'H\^{o}pital rule and considering
\eqref{derpsi}, we obtain that, for any $A>0$,
\begin{align*}
& \limsup_{m \to \infty}\frac{1}{\Delta_t^{(m)}} \\
&
\leq \frac{1}{C_{\alpha, t}}
\lim_{m \to \infty}
\Big [2Am+1-2AmC_{\alpha} \Psi
\Big (\frac{1-\alpha}{2}, 1-\alpha ; \log
\Big (\frac{Am+1}{Am}
\Big )
\Big )
\Big ] \notag\\
&=\lim_{m \to \infty} \frac{2A+\frac{1}{m}-2AC_{\alpha}
 \Psi
\Big (\frac{1-\alpha}{2}, 1-\alpha ; \log
\Big (\frac{Am+1}{Am}
\Big )
\Big )}{1/m} \notag\\
&=\lim_{m \to \infty}
\frac{\frac{1}{m^2}+A(\alpha -1)
C_{\alpha}\frac{Am}{Am+1}\frac{-A}{A^2 m^2}
 \Psi
\Big (\frac{3-\alpha}{2}, 2-\alpha ;
 \log
\Big (\frac{Am+1}{Am}
\Big )
\Big )}{1/m^2} \notag \\
&=\lim_{m \to \infty}
\Big [1-\frac{Am}{Am +1}
(\alpha -1)C_{\alpha}\Psi
\Big (\frac{3-\alpha}{2}, 2-\alpha ; \log
\Big (\frac{Am+1}{Am}
\Big )
\Big )
\Big ]=0, \notag
\end{align*}
where, in the last step, we have
applied \eqref{lim} and the duplication
formula of the gamma function.
Therefore, we have that
$\lim_{m \to \infty} \Delta_t ^{(m)} = \infty $
and thus, in this case, the
process displays long memory.
\end{proof}

\begin{remark}\rm
    By considering the previous result,
we can compare the long-time behavior
of the H-fBm to the fBm itself
(herafter denoted as
$B_\alpha
:=\left\{ B_\alpha (t)\right\}_{t \geq 0}$):
let
$\delta_t ^{(m)}$
be the ratio defined in \eqref{delta} with
$X_\alpha^H(\cdot)$ replaced by the
increment of the fBm, i.e., $X_\alpha(j)
:=B_\alpha(j)-B_\alpha(j-1)$.
Then, it is easy to check that
$\delta_t ^{(m)}=m^{\alpha-1}$.

For $\alpha\in (0,1)$, we obtain from
\eqref{lrd}, \eqref{lrd2}, and \eqref{lrd3}
 that
\begin{align*}
\frac{\delta_t ^{(m)}}{\Delta_t^{(m)}}
& =  \frac{m^{\alpha-2}}{C_{\alpha, t}}
 \Big [m^2
\Big  (2t-1
\Big )-2C_\alpha S_{t,m}
\Big ]  \notag \\
&\geq   \frac{m^\alpha (2t-1)}{C_{\alpha, t}}
 - \frac{m^{\alpha}}{C_{\alpha, t}}
 \Big [
 2t-1  -\frac{1}{m}
\Big ] C_{\alpha}\Psi
\Big  (\frac{1-\alpha}{2}, 1-\alpha  ; \log
\Big  (\frac{tm}{ tm-1 }
\Big  )
\Big  ), \notag
\end{align*}
for any $t>1$.
By recalling \eqref{asy1}, we can write that
\[
C_{\alpha}\Psi
\Big (\frac{1-\alpha}{2}, 1-\alpha  ; \log
\Big (\frac{tm}{ tm-1 }
\Big )
\Big )=1+O((\log(tm/(tm-1))^\alpha),
\]
as $m \to \infty $,
so that
$\liminf_{m \to \infty}
\delta_t ^{(m)}/\Delta_t^{(m)} \geq K$,
where $K$ is a positive constant,
by \eqref{ub1}. We can then state
that the anti-persistence of the
H-fBm is at least not weaker (or stronger)
than in the case of the standard fBm,
in the sense that the variance of the
increments' sum increases not more
rapidly (and possibly more slowly)
in our case. We leave it for future
research to investigate whether the
lower bound that we derived is sharp.

On the other hand, for
$\alpha\in (1,2)$,
by considering that
$\Psi(a,b;\log(x/(x-1)))$
is non-increasing, for $a<0$,
we obtain from \eqref{lrd3} that
\[S_{t,m}\leq \frac{1}{2}\left[m^2
  (2t-1
  )-m\right] \Psi
\Big (\frac{1-\alpha}{2}, 1-\alpha  ;
 \log
\Big (\frac{tm-m+1}{ tm-m }
\Big )
\Big ), \quad t>1,
\]
so that
\begin{align*}
\frac{\delta_t ^{(m)}}{\Delta_t^{(m)}}
& =
\frac{m^{\alpha-2}}{C_{\alpha, t}}  \left[m^2
\left (2t-1
\right )-2C_\alpha S_{t,m}\right]  \notag \\
& \geq
\frac{m^\alpha (2t-1)}{C_{\alpha, t}} \\
&
\quad
- \frac{m^{\alpha}}{C_{\alpha, t}}
\Big [
 2t-1  -\frac{1}{m}
 \Big ] C_{\alpha}\Psi
\Big (\frac{1-\alpha}{2}, 1-\alpha  ;
 \log
\Big (\frac{(t-1)m+1}{(t-1)m}
\Big )
\Big ).
\end{align*}
By taking into account \eqref{asy0},
we get that
$\liminf_{m \to \infty}
\delta_t ^{(m)}/\Delta_t^{(m)} \geq +\infty$,
thus showing that the long-range
dependence is weaker than in the
case of the fBm; to be more precise,
the variance of the increments'
sum of our process diverges with
a slower rate than for fBm.
\end{remark}

As a consequence of
Def.\ \ref{def3.1} and of Lemma \ref{lem2.1},
we can give the following, finite-dimensional,
representation of H-fBm,  in terms of a
stochastic integral, which is the
analogue of the Mandelbrot-Van Ness
representation for the fBm:
  \begin{align}
      B_{\alpha}^H (t)
      & =
\frac{1}{\sqrt{\Gamma(\alpha)}}
\int _{0}^{\infty}
\Big ( \log \frac{t}{s}
\Big )_{+}^{ { \frac{\alpha-1 }{2} } }dB(s) \notag \\
&
=\frac{1}{\sqrt{\Gamma(\alpha)}}\int _{0}^{t}
\Big ( \log \frac{t}{s}
\Big )^{ { \frac{\alpha-1 }{2} } }dB(s), \label{mvn}
  \end{align}
  where
$\left\{B(t)\right\}_{t \geq 0}$ is
a standard Brownian motion.
The last integral is well-defined
by considering Lemma \ref{lem2.1}
and the equality in distribution to
the H-fBm can be easily checked
by the It\={o}-isometry and recalling
Theorem \ref{proc}.

\begin{remark}\rm
    The representation
\eqref{mvn} provides a useful tool in order
to simulate the H-fBm's trajectories.
Indeed, analogously to the fBm case,
we can simulate the H-fBm at points
$0 \leq t_1 \leq t_2 \leq ...\leq t_n =T$,
for
$n \in \mathbb{N}$, by the following procedure:
we first build a vector of
$n$ numbers
drawn according to a standard
Gaussian distribution; then we
multiply it component-wise by
$\sqrt{T}/n$
to obtain the increments of a standard Bm on
$\left[0, T \right]$, given by the vector
$
\left ( \Delta B_1, ..., \Delta B_n
\right )$. For each $t_j$ we compute
\[
\hat{B}_ \alpha ^{H}(t_{j})=
{\frac {n}{T}}\frac{1}
{\sqrt{\Gamma(\alpha)}}
\sum _{i=0}^{j-1}
\int _{t_{i}}^{t_{i+1}}
K_{H}(t_{j},\,s)\,ds\ \Delta B_{i},
\]
for
$$
K_H(t,s)
:=
\Big ( \log \frac{t}{s}
\Big )^{(\alpha -1)/2} .
$$
Finally, the integral may be
efficiently computed by the
Gaussian quadrature method.

\end{remark}

  \begin{remark}\rm
  We notice that the processes defined as
$$
\int _{0}^{\infty} k(s,t)dB(s) ,
$$
  either for the kernel
$k(s,t)=K_\alpha(t-s)^{ { \frac{\alpha-1 }{2} } }$,
with
$\alpha \in (0,2)$ (fBm case) or
for its generalization
$k_{\varphi}(s,t)= K_{\varphi} \nu(t-s),$
where
 $\nu(\cdot)$
is the tail of a  L\'{e}vy measure with
Laplace exponent
$\varphi(\cdot)$
(or its Sonine associate kernel, see
\cite{BCM}),
always possess stationary increments on
$(s,t)$, for any $s,t>0$, since the
kernel is expressed by
means of the difference $t-s$
(time-homogeneous kernel). In this case,
on the contrary, it is evident from
\eqref{mvn} that the distribution of
$B_{\alpha}^H (t)-B_{\alpha}^H (s)$
depends on the ratio $t/s$,
analogously to what happens for the
covariance of the process (see formula \eqref{cov}).
  \end{remark}

\section{Le Roy measure}

Let us denote the \textit{Le Roy function} by
$$
\mathcal{R}_{\beta
}(x)
:=\sum_{j=0}^{\infty }x^{j}/(j!)^{\beta },
$$
which is defined for
any
$x\in \mathbb{C},$
$\beta >0.$ Clearly, for
$\beta =1$, the Le Roy
function reduces to the exponential function.
In view of what follows, we recall that a function
$
g:[0,\infty )\rightarrow \lbrack 0,\infty )$
 is completely monotone if
$
(-1)^{n}g^{(n)}(x)\geq 0$, for any
$x\geq 0,$
$n\in \mathbb{N},$ and that
$
\mathcal{R}_{\beta }(-s)$,
$s>0$, is completely monotone, for any
$\beta \in
(0,1]$ (see e.g.\
\cite{BOU},
\cite{SIM}). Thus, we will hereafter restrict to
this interval for the parameter
$\beta .$
Thanks to the complete monotonicity and
 by considering the Bernstein theorem
(see
\cite{SCH}, p.\ 3), there exists a measure
$\mu _{\beta }:\mathbb{R}
_{+}\rightarrow \lbrack 0,1],$ such that
\begin{equation}
\mathcal{R}_{\beta }
\left ( -s
\right ) =\int_{\mathbb{R}_{+}}e^{-st}d\mu
_{\beta }(t),\qquad s>0.  \label{mea}
\end{equation}

Moreover, by applying the results given in
\cite{BER}, we know that for
$\beta \in (0,1)$, the
measure
$\mu _{\beta }$ is absolutely continuous
 with density function given by
 the inverse Mellin transform of
\begin{equation}
\int_{0}^{\infty }
t^{s}m_{\beta }(t)dt=\Gamma(s+1)^{1-\beta },
\qquad s > -1 \notag
,
\end{equation}
that is,
\begin{equation}
m_{\beta }(t)=\frac{1}{2\pi }
\int_{-\infty }^{\infty }t^{ix-1}\Gamma
(ix-1)^{1-\beta }dx,\qquad t>0  \label{mea2}
\end{equation}
on the positive half line.
Plainly, the
$r$-th moment is given by
\begin{equation}
\int_{0}^{\infty }
t^{r}m_{\beta }(t)dt=(r!)^{1-\beta },
\qquad r\in \mathbb{N
}.  \label{ber}
\end{equation}
It is evident from \eqref{mea}
 that for
$\beta=1$ the density
$\mu_\beta(\cdot)$ coincides
with the Dirac delta distribution
at one. Hereafter, we will consider
$\beta \in (0,1)$, while, in
some remarks, we will refer to the limiting case
$\beta=1$.

\subsection{Le Roy measure on
$\mathbb{R}^n$}
As a consequence of (\ref{mea}),
we can apply the Bochner theorem (see
\cite{REE}) and give the following

\begin{definition}
Let
$\beta \in (0,1)$.
We define the
\emph{$n$-dimensional Le Roy measure}
$\nu
_{\beta }^{n}(\cdot )$
as the unique probability measure on
$(\mathbb{R}^{n},
\mathcal{B}(\mathbb{R}^{n}))$
that satisfies:
\begin{equation}
\mathcal{R}_{\beta }
\Big ( -\frac{
 \langle\xi ,\xi \rangle}{2}
\Big ) =\int_{
\mathbb{R}^{n}}e^{i
 \langle x,\xi \rangle}
 d\nu _{\beta }^{n}(x),\qquad \xi
\in \mathbb{R}^{n}.  \label{ch}
\end{equation}
Moreover, let
$
\left ( X_{\beta ,1},...,X_{\beta ,n}
\right )
$ be the random
vector (with values in
$\mathbb{R}_{+}^{n}$) with joint distribution
$\mu
_{\beta }^{n}(\cdot )$ and characteristic function
$$
\Phi _{\beta }(\xi
_{1},...,\xi _{n})
:=\mathcal{R}_{\beta }
\left ( -
 \langle\xi ,\xi
\rangle/2
\right ) ,
\ \ \ \ \xi \in \mathbb{R}_{+}^{n}.
$$
\end{definition}

\begin{lem}
\label{lem1} The mixed moments of orders
$r_{1},...,r_{n}\in \mathbb{N}$ of
the random vector
$
\left ( X_{\beta ,1},...,X_{\beta ,n}
\right ) $ are
\begin{align}
M_{r_{1},...,r_{n}} &
:=\mathbb{E}
\left[ X_{\beta ,1}{}^{r_{1}}\cdot \cdot
\cdot X_{\beta ,n}{}^{r_{n}}
\right]  \label{mm} \\
&=\left\{
\begin{array}{l}
0,\qquad
\text{\ for at least one }r_{j}=2m_{j}+1 \\
2^{-m}(m!)^{1-\beta }
\prod\limits_{j=1}^{n}\frac{(2m_{j})!}{m_{j}!}
,\qquad \text{\ for }r_{j}=2m_{j},\;j=1,...,n
\end{array}
\right.  \notag
\end{align}
where
$m_{j}=1,2,...$, and
$m=\sum_{j=1}^{n}m_{j}$.
\end{lem}

\begin{proof}[\bf Proof]
We start by proving that the following
equality of the finite dimensional
distributions holds
\begin{equation}
\left ( X_{\beta ,1},...,X_{\beta ,n}
\right ) \overset{f.d.d.}{=}
\left ( \sqrt{
Y_{\beta }}X_{1},...,\sqrt{Y_{\beta }}X_{n}
\right ) ,  \label{fdd}
\end{equation}
where
$Y_{\beta }$ is a random variable with distribution
$P(Y_{\beta }\in
B)=\mu _{\beta }(B),$ for any
$B\in \mathcal{B}(\mathbb{R}_{+})$,
independent from the standard Gaussian vector
$
\left ( X_{1},...,X_{n}
\right )
$ and
$\mu _{\beta }$ is the measure
defined in (\ref{mea}). Indeed, by
conditioning and considering
(\ref{ch}), we have that
\begin{align*}
&
\mathbb{E}\exp
\Big \{ i\xi _{1}\sqrt{Y_{\beta }}
X_{1}+...+i\xi _{n}\sqrt{
Y_{\beta }}X_{n}\Big \}  \\
& =
 \mathbb{E}_{\mu _{\beta }}
\Big  ( \exp \Big \{ -
\frac{1}{2}
\Big  ( \xi _{1}^{2}Y_{\beta }
+...+\xi _{n}^{2}Y_{\beta }
\Big  )
\Big \}
\Big  )  = \mathcal{R}_{\beta }
\Big  ( -\frac{
 \langle\xi ,\xi \rangle}{2}
\Big  )
=\Phi _{\beta }(\xi _{1},...,\xi _{n}).
\end{align*}
Then, according to (\ref{fdd}),
we can easily derive the moments in (\ref{mm}
), by applying (\ref{ber}) and
recalling that, for a standard Gaussian vector
the odd order moments are null, while
$$
\mathbb{E}\left[ X_{1}^{2m_{1}}\cdot
\cdot \cdot X_{n}^{2m_{n}}\right] =2^{-m}\prod
\limits_{j=1}^{n}(2m_{j})!/m_{j}! ,
$$
 where
$m=\sum_{j=1}^{n}m_j$. Hence, we get
\begin{equation*}
M_{r_{1},...,r_{n}}=
\prod\limits_{j=1}^{n}\mathbb{E}\left[ X_{j}^{2m_{j}}
\right] \mathbb{E}\left[ Y_{\beta }^{m_{j}}\right] ,
\end{equation*}
where we have also considered the independence among
$X_{1},...,X_{n},Y_{
\beta }.$
\end{proof}

We now prove that the measure
$\nu _{\beta }^{n}(\cdot )$ can be obtained as
product measure of
$\nu _{\beta }^{k}(\cdot )$ and
$\nu _{\beta }^{l}(\cdot
)
$ for
$k,l\in \mathbb{N}$ and such that
$k+l=n,$ only in the limiting case
$\beta =1$
. To this aim we start by evaluating
the first Hermite polynomials
$
H_{j}^{\beta },$
$j=0,1,2,3,4$, orthogonal in
$L^{2}(\mathbb{R},\nu _{\beta
}^{1})$ and with
$\deg H_{j}^{\beta }=j$.
To this aim, we solve the
system of equations
$$
\mathbb{E}\left[ X_{\beta
}^{k}(a_{0}+a_{1}X_{\beta }+...+
a_{j-1}X_{\beta }^{j-1}+X_{\beta }^{j})
\right] =0 ,
\ \text{ for
$k=0,1,...,j$.}
$$
 By considering
(\ref{mm}), we obtain that
\begin{eqnarray}
H_{0}^{\beta }(x)
&\equiv &1,\qquad H_{1}^{\beta }(x)=x,
 \label{her} \\
H_{2}^{\beta }(x)
&=&x^{2}-1,\qquad
 H_{3}^{\beta }(x)=x^{3}-\frac{3!}{
(2!)^{\beta }}x,  \notag \\
H_{4}^{\beta }(x)
 &=&x^{4}+a_{2}x^{2}+a_{0},  \notag
\end{eqnarray}
where
\begin{equation*}
a_{0}=-\frac{6}{2^{\beta }}-
\frac{6-90\cdot 3^{-\beta }}{6-2^{\beta }}
,\qquad a_{2}=\frac{6-90\cdot
3^{-\beta }}{6-2^{\beta }}.
\end{equation*}
As a check, we can see that in
 the limiting case
$\beta =1,$ we obtain the
well-known first five Hermite
polynomials that hold for the Gaussian
measure. It is immediate from
(\ref{her}) that
\begin{align*}
& \int_{\mathbb{R}^{2}}H_{4}^{\beta }
(x_{1})H_{2}^{\beta }(x_{2})d\nu
_{\beta }^{2}(x)
=\int_{\mathbb{R}^{2}}(x_{1}^{4}
+a_{2}x_{1}^{2}+a_{0})(x_{2}^{2}-1)d\nu
_{\beta }^{2}(x) \\
&=\frac{(18-6\cdot 3^{\beta })
(6-2^{\beta })+(2\cdot 3^{\beta }-6^{\beta
})(6-90\cdot 3^{-\beta })}
{6^{\beta }(6-2^{\beta })}=0,
\end{align*}
if and only if
$\beta =1.$

\subsection{Le Roy grey noise space}
Let now denote by
$\mathcal{S}
:=\mathcal{S}(\mathbb{R)}$ the Schwartz space
of the infinitely differentiable,
 rapidly decreasing functions and by
$
\mathcal{S}^{\prime }
:=\mathcal{S}^{\prime }(\mathbb{R)}$
 its dual. Then it
is well-known that
$\mathcal{S}\subset L^{2}
(\mathbb{R},dx)\subset \mathcal{S
}^{\prime }$ is a nuclear
triple and we can define the measure
$\nu _{\beta
}
$ on
$(\mathcal{S}^{\prime },\mathcal{B})$
 by considering
the Bochner-Minlos theorem, where
$\mathcal{B}$ is the
$\sigma
$-algebra generated by the cylinders
\cite{HID}.

\begin{definition}\label{cfc}
Let
$\beta \in (0,1)$, we define the
infinite-dimensional Le Roy measure
$
\nu _{\beta }(\cdot )$ as the unique
probability measure such that
\begin{equation}
\int_{\mathcal{S}^{\prime }}
e^{i\left
 \langle x,\xi \right\rangle }d\nu
_{\beta }(x)=\mathcal{R}_{\beta }
\Big ( -\frac{
 \langle \xi ,\xi \rangle }{2}
\Big ) \text{,\qquad }\xi \in \mathcal{S}.
 \label{cfceq}
\end{equation}
We call
$(\mathcal{S}^{\prime },
\mathcal{B},\nu _{\beta })$
the \emph{Le Roy
grey noise space} and we denote by
$L^2(\nu_\beta)$
the corresponding Hilbert space
$L^2(\mathcal{S}^{\prime },
\mathcal{B},\nu _{\beta })$.
\end{definition}

We denote by
$
 \langle \cdot ,\cdot \rangle
$ not only the inner product in
$
L^{2}(\mathbb{R},dx)\times L^{2}
(\mathbb{R},dx)$, but also the dual pairing
on
$\mathcal{S}^{\prime }
(\mathbb{R})\times \mathcal{S}(\mathbb{R})$.
Moreover, we consider its extension to
$\mathcal{S}^{\prime }(\mathbb{R}
)\times L^{2}(\mathbb{R},dx).$
Finally, let
$SP$ denote the set of the special permutations
\begin{equation*}
(1,2,...,2m)\rightarrow (r_{1},
s_{1},...r_{m},s_{m})
\end{equation*}
such that
$r_{1}<r_{2}<...<r_{n},$ and
$r_{j}<s_{j}$, for
$j=1,2,...,m$ (see
\cite{SCHN}).

\begin{lem}\label{lem3.2}
The moments of the measure
$\nu _{\beta }$ are given by
\begin{eqnarray}
&&\int_{\mathcal{S}^{\prime }(\mathbb{R})}\langle u,\xi \rangle ^{2m}d\nu
_{\beta }(u)=\frac{(2m)!}{2^{m}(m!)^{\beta }}\langle \xi ,\xi \rangle
^{m},\qquad \int_{\mathcal{S}^{\prime }(\mathbb{R})}\langle u,\xi \rangle
^{2m+1}d\nu _{\beta }(u)=0,  \label{ss} \\
&&\int_{\mathcal{S}^{\prime }(\mathbb{R})}\prod\limits_{j=1}^{2m}\langle
u,\xi _{j}\rangle d\nu _{\beta }(u)=(m!)^{1-\beta
}\sum_{SP}\prod\limits_{k=1}^{m}\langle \xi _{r_{k}},\xi _{s_{k}}\rangle
,\qquad \int_{\mathcal{S}^{\prime }(\mathbb{R})}\prod\limits_{j=1}^{2m+1}%
\langle u,\xi _{j}\rangle d\nu _{\beta }(u)=0, \notag \\
&&  \qquad  \label{ss2}
\end{eqnarray}%
for
$m\in \mathbb{N}$,
$\xi ,\xi _{i} \in \mathcal{S}(\mathbb{R}),$
$
i\in \mathbb{N}$.
\end{lem}

\begin{proof}[\bf Proof]
Formula (\ref{ss}) can be easily
obtained by considering Lemma \ref{lem1},
while for (\ref{ss2}), for
$m=1$, according to (\ref{cfc}),
 we can write that
\begin{align*}
&
\int_{\mathcal{S}^{\prime }(\mathbb{R})}
 \langle u,\xi _{1}\rangle
 \langle
u,\xi _{2}\rangle d\nu _{\beta }(u) \\
& =i^{-2}\frac{\partial ^{2}}{\partial
a_{1}\partial a_{2}}\mathcal{R}_{\beta }
\Big ( -\frac{
 \langle a_{1}\xi
_{1}+a_{2}\xi _{2},a_{1}\xi _{1}
+a_{2}\xi _{2}\rangle }{2}
\Big ) \bigg\rvert
_{a_{1}=a_{2}=0} \\
& =-\frac{\partial ^{2}}{\partial
a_{1}\partial a_{2}}
\mathcal{R}_{\beta}
\Big ( -\frac{a_{1}^{2}\lVert
 \xi _{1}\rVert ^{2}+a_{2}^{2}\lVert \xi
_{2}\rVert ^{2}+2a_{1}a_{2}
 \langle \xi _{1},\xi _{2}\rangle }{2}
\Big )
\bigg\rvert_{a_{1}=a_{2}=0} \\
& =: -\frac{\partial ^{2}}
{\partial a_{1}\partial a_{2}}\mathcal{R}_{\beta
}
\Big ( -\frac{A_{a_{1},a_{2}}}{2}
\Big ) \bigg\rvert_{a_{1}=a_{2}=0}.
\end{align*}
By considering that
$$
\frac{d}{dx}\mathcal{R}_{\beta
}(-x)=-\sum_{l=0}^{\infty }
\frac{(l+1)^{1-\beta }(-x)^{l}}{(l!)^{\beta }}
$$
and that the Le Roy function is
entire (so that the interchange of sum and
derivative is allowed, see
\cite{BOU}), we obtain
\begin{align*}
&
\int_{\mathcal{S}^{\prime }(\mathbb{R})}
 \langle u,\xi _{1}\rangle
 \langle
u,\xi _{2}\rangle d\nu _{\beta }(u) \\
 & =\frac{\partial }{\partial a_{1}}
\Big [
(a_{2}\lVert \xi _{1}\rVert ^{2}+a_{1}
 \langle \xi _{1},\xi _{2}\rangle
)\sum_{l=0}^{\infty }\frac{(l+1)^{1-\beta }
(-A_{a_{1},a_{2}})^{l}}{
(l!)^{\beta }} \Big ] \\
& =
 \langle \xi _{1},\xi _{2}\rangle
  \sum_{l=0}^{\infty }\frac{(l+1)^{1-\beta
}(-A_{a_{1},a_{2}})^{l}}{(l!)^{\beta }}
 +(a_{2}\lVert \xi _{1}\rVert
^{2} \\
& \quad
+a_{1}
 \langle \xi _{1},\xi _{2}\rangle )
 \sum_{l=1}^{\infty }\frac{
(l+1)^{1-\beta }(-A_{a_{1},a_{2}})^{l-1}}
{(l!)^{\beta }}\frac{\partial }{
\partial a_{1}}A_{a_{1},a_{2}},
\end{align*}
which coincides with (\ref{ss2}) with
$m=1,$ as, for
$a_{1}=a_{2}=0$ we
have that
$A_{a_{1},a_{2}}=0$ and
$\frac{\partial }{\partial a_{1}}
A_{a_{1},a_{2}}=0$. By means of a
similar reasoning, we obtain (\ref{ss2}),
 for
$
m\geq 2.$
\end{proof}

\begin{remark}\rm
We note that the covariance given in
formula (\ref{ss2}), for
$m=1,$ is not
affected by the parameter
$\beta
$ and thus it coincides with that obtained
under a Gaussian measure.
This result differs to what happens in the case of
other grey noise measures, for
example in the Mittag-Leffler case
(see \cite{GRO}) or in the incomplete-gamma case (see
\cite{BEG}).
\end{remark}

\subsection{Spaces of test functions and
distributions for the Le Roy measure}

In this subsection, we prove the existence of
test functions and distributions spaces for
$\nu_\beta$ and we establish the
characterization theorems and tools
for the analysis of the corresponding
distribution spaces.
We follow the theory given in
\cite{KSWY98}.
In order to build the latter spaces,
we show the analyticity of the Le Roy
measure on
$(\mathcal{S}', \mathcal{B})$
by proving the following properties:

\begin{enumerate}
\item[A1)] For
$\beta \in (0,1)$,
$\nu_{\beta}$ has an
analytic Laplace transform in a neighborhood of zero
$\mathcal{U} \subset \mathcal{S}_{
\mathbb{C}}$:
\begin{equation}\label{propertyA1}
\mathcal{S}_{\mathbb{C}}
\supset \mathcal{U} \ni \phi
 \mapsto \ell_{\nu_{\beta}} (\phi)
:=\int_{\mathcal{S
}^{\prime}}\exp{
 \langle \omega,\phi
 \rangle }d\nu_{\beta}(\omega)=\mathcal{R}_\beta
 \Big (\frac{
 \langle \phi, \phi
 \rangle}{2}
\Big ).
\end{equation}

\item[A2)] \label{propertyA2} For
$\beta \in (0,1)$,
$\nu_{\beta}(\mathcal{U}
)>0$ for any non-empty open subset
$\mathcal{U}\subset \mathcal{S}^{\prime}$.
\end{enumerate}

We now prove Property A1 by showing that,
for the measure defined in (\ref{mea}),
the Laplace transform is well-defined and
that it is holomorphic. To this aim,
following the complexification procedure
of a real Hilbert space as a direct sum, we define
$$
\mathcal{S
}_{\mathbb{C}}
:=\mathcal{S}\oplus i \mathcal{S}=\{\xi_1 + i \xi_2 |
\xi_1,\xi_2 \in \mathcal{S}\}
$$
 and the bilinear extension of the scalar product in
$\mathcal{S}$ as
$
 \langle \xi,\phi \rangle_{\mathcal{S
}_{\mathbb{C}}}
:=
 \langle \overline{\xi},\phi \rangle_\mathcal{S
}
$ (for further details, see
\cite{DIN}).

\begin{lem}
\label{Le1} Let
$\beta \in (0,1)$ and
$\lambda \in \mathbb{R}/\{0\}$, then
the exponential function
$\mathcal{S}^{\prime }\ni x \mapsto
e^{|\lambda
 \langle x,\phi \rangle |}$ is
 integrable with respect to
$\nu _{\beta
}(\cdot)$, and
\begin{equation}
\ell_{\nu_\beta}(\lambda \phi)
:=\int_{\mathcal{S}^{\prime }}e^{\lambda
 \langle x,\phi \rangle }d\nu _{\beta
}(x)=\mathcal{R}_{\beta }
\Big ( \frac{\lambda ^{2}
 \langle \phi, \phi \rangle}{2}
\Big )
,\qquad \text{ for }\phi \in \mathcal{S}_{\C},  \label{le}
\end{equation}
is holomorphic in
$\mathcal{S}_{\mathbb{C}}$.
\end{lem}

\begin{proof}[\bf Proof]
We start by proving the integrability, for
$\lambda \in \mathbb{R}/\{0\}$.
We can define the monotonically increasing sequence
$$
g_{N}(\cdot )
:=\sum_{n=0}^{N}
\frac{1}{n!}|
 \langle \cdot ,\lambda \phi \rangle |^{n} .
 $$
 We divide this sum
into odd and even terms,
\begin{align*}
g_{N}(x)& =\sum_{n=0}^{\lfloor N/2\rfloor }\frac{1}{(2n)!}|
 \langle x,\lambda
\phi \rangle |^{2n}+
\sum_{n=0}^{\lceil N/2\rceil -1}\frac{1}{(2n+1)!}
|
 \langle x,\lambda \phi \rangle |^{2n+1}\\
&=:\sum_{n=0}^{\lfloor N/2\rfloor
}O_{2n}(x)+\sum_{n=0}^{\lceil N/2\rceil -1}O_{2n+1}(x).
\end{align*}
For the even terms we get from (\ref{ss}) that:
\[
\int_{\mathcal{S}^{\prime }}O_{2n}(x)
d\nu _{\beta }(x)=\frac{(\lambda
^{2}
 \langle \phi, \phi
  \rangle/2)^{n}}{(n!)^{\beta }}=:E_{n}.
\]
We estimate the odd terms using the
Cauchy-Schwarz inequality on
$L^{2}(
\mathcal{S}^{\prime },$ $\mathcal{B},
\nu _{\beta })$ and the inequality
$
st\leq 1/2(s^{2}+t^{2})$, for
$s,t\in \mathbb{R}$:
\begin{align*}
&\int_{\mathcal{S}^{\prime }}
O_{2n+1}(x)d\nu _{\beta }(x) \\
& = \frac{1}{(2n+1)!}
\int_{\mathcal{S}^{\prime }}|
 \langle x,\lambda \phi
\rangle |^{n+1}|
 \langle x,\lambda \phi
  \rangle |^{n}d\nu _{\beta }(x) \\
&\leq \frac{1}{(2n+1)!}
\Big(\int_{\mathcal{S}^{\prime }}|
 \langle x,\lambda
\phi \rangle |^{2n+2}
d\nu _{\rho ,\theta }(x)
\Big)^{1/2}\Big(\int_{\mathcal{S
}^{\prime }}|
 \langle x,\lambda \phi
 \rangle |^{2n}d\nu _{\rho ,\theta }(x)
\Big)^{1/2} \\
&\leq \frac{1}{2}
\Big(\int_{\mathcal{S}^{\prime }}\frac{|
 \langle x,\lambda
\phi \rangle |^{2n+2}(2n+2)}{(2n+2)!}
d\nu _{\rho ,\theta }(x)+\int_{\mathcal{
S}^{\prime }}\frac{|
 \langle x,\lambda \phi
 \rangle |^{2n}}{(2n)!}d\nu _{\rho
,\theta }(x)\Big) \\
& = (n+1)\int_{\mathcal{S}^{\prime }}
O_{2n+2}(x)d\nu _{\beta }(x)+\frac{1}{2}
\int_{\mathcal{S}^{\prime }}
O_{2n}(x)d\nu _{\beta }(x) \\
&\leq (n+1)\frac{(\lambda ^{2}
 \langle \phi, \phi
 \rangle/2)^{n+1}}{((n+1)!)^{\beta }}+
\frac{(\lambda ^{2}
 \langle \phi, \phi
 \rangle/2)^{n}}{(n!)^{\beta }}
 =(n+1)E_{n+1}+E_{n}.
\end{align*}
Considering the odd and even terms
together, we can write that
\[
\int_{\mathcal{S}^{\prime }}g_{N}(x)
d\nu _{\beta }(x)\leq E_{\lfloor
N/2\rfloor }+2
\sum_{n=0}^{\lceil N/2\rceil -1}E_{n}
+\sum_{n=0}^{\lceil
N/2\rceil -1}(n+1)E_{n+1}<\infty .
\]
Since, for any
$\phi (\cdot )$ with
$\Vert \phi \Vert ^{2}<\infty ,$
$
\lim_{N\rightarrow \infty }
E_{\lfloor N/2\rfloor }=0$ and
\[
\lim_{N\rightarrow \infty }
\sum_{n=0}^{\lfloor N/2\rfloor }E_{n}
=\mathcal{R}
_{\beta }
 \Big ( \frac{\lambda ^{2}
 \langle \phi, \phi \rangle}{2}
\Big ) ,
\]
we conclude that
$$
\lim_{N\rightarrow \infty }
\int_{\mathcal{S}^{\prime
}}g_{N}(x)d\nu _{\beta }(x)<\infty ,
$$
by applying the monotone converge
theorem and the ratio criterion (since
$\lim_{n\rightarrow \infty
}(n+1)E_{n+1}/nE_{n}=0,$ for
$\beta <1$). Equation (\ref{le}) simply
follows.

Finally, it is easy to check that
$\ell_\beta(\cdot)$ is a continuous
function (following the same
lines of the proof of Theorem 2.1 in
\cite{BEG}).
Then, in order to prove that it is
holomorphic, by the Morera's theorem,
it is enough to prove that,
for any closed and bounded curve
$\gamma\in \mathbb{C}$,
$\int_{\gamma}\ell_\beta(z)dz=0$.
Indeed, this holds in view of the
 Fubini's theorem,
\begin{equation*}
\int_{\gamma}\int_{\mathcal{S}^{\prime}}e^{
 \langle x, \xi + z\eta
\rangle}d\nu_{\beta}(x)dz=
\int_{\mathcal{S}^{\prime}}\int_{\gamma}e^{
\langle x, \xi + z\eta \rangle}dz d\nu_{\beta}(x)=0
\end{equation*}
as the exponential function is holomorphic.
\end{proof}

In order to verify that Property A2
is satisfied by
$\nu_{\beta}$, we prove that, for
$\beta \in (0,1)$,
$\nu_{\beta}$ is always strictly
positive on non-empty, open
subsets, by resorting to their
representation as mixture of
Gaussian measures.

\begin{theorem}\label{thm:nonnegativmeasure}
For any open, non-empty set
$\mathcal{U}\subset \mathcal{S}^{\prime}$ and
for any
$\beta \in (0,1)$,
we have that
$
\nu_{\beta}(\mathcal{U})>0$.
\end{theorem}

\begin{proof}[\bf Proof]
It is sufficient to prove that
$
\nu_{\beta}$ is an elliptically
contoured measure, i.e.,
 if we denote
by
$\nu^{s}$ the centered Gaussian measure on
$\mathcal{S}'$ with
variance
$s>0$, the following holds:
\begin{equation}  \label{Gaussianity}
	\nu_{\beta}= \int_{0}^{\infty}
	 \nu^{s} \, \mathrm{d} \mu_{\beta}(s),
\end{equation}
where
$\mu_{\beta}$ is the measure defined on
$(0,\infty)$ by \eqref
{mea}. The identity in
equation~\eqref{Gaussianity}
can be checked by considering
that
\begin{equation*}
	\int_{\mathcal{S}'}\e^{\mathrm{i}
 \langle \omega, \xi \rangle}
  \mathrm{d}\nu_s(\omega)=\exp
\left (
	-\frac{s}{2}
 \langle \xi, \xi \rangle
\right ), \quad \xi \in \mathcal{S}
\end{equation*}
and thus, by \eqref{mea},
\begin{equation}
	\int_{0}^{\infty} \exp
\Big ( -\frac{s}{2}
 \langle \xi, \xi \rangle
\Big ) \mathrm{d}
	\mu_{\beta}(s)=\mathcal{R}_\beta
\Big (-\frac{1}{2}
 \langle \xi,
		\xi \rangle
\Big ),  \label{ec}
\end{equation}
which coincides with
$\int_{\mathcal{S}'} \e^{\mathrm{i}
 \langle \omega, \xi
	\rangle}\mathrm{d}\nu_{\beta}(\omega)$.
\end{proof}

By Lemma 4.12, Sec.\ 5, and Sec.\ 6 in
\cite{KSWY98}, the test function space, i.e.,
$(\mathcal{S})^{1}_{\nu_{\beta}}$,
and the distribution space, i.e.,
$(\mathcal{S})^{-1}_{\nu_{\beta}}$,
exist and we have:
\[  (\mathcal{S})^{1}_{\nu_{\beta}}
\subset L^2(\nu_{\beta})
\subset(\mathcal{S})^{-1}_{\nu_{\beta}}\]
endowed with the dual pairing
$
 \langle\!
 \langle \cdot,\cdot \rangle\!
 \rangle_{\nu_{\beta}}$ between
$ (\mathcal{S})^{-1}_{\nu_{\beta}}
$ and
$(\mathcal{S})^{1}_{\nu_{\beta}}$,
 which is the bilinear extension
  of the inner product of
$L^2(\nu_{\beta})$.

We define the
$S_{\nu_{\beta}}$-transform by means of
the normalized exponential
$e_{\nu_{\beta}}(\cdot, \xi)$:
\begin{align*}
S_{\nu_{\beta}}(\Phi )(\xi)
& :=
 \langle\!
 \langle \Phi, e_{\nu_\beta}(\cdot, \xi)
  \rangle\!\rangle_{\nu_{\beta}} \\
  &
:=\frac{1}{\mathcal{R}_\beta
\left ( \frac{1}{2}
 \langle \xi,\xi\rangle
\right )}\int_{S'}e^{
 \langle \omega, \xi\rangle}
 \Phi(\omega)\nu_{\beta}(d\omega),
  \quad \xi \in  U_{p,q},
\end{align*}
for
$\Phi \in (\mathcal{S})^{-1}_{\nu_{\beta}}$ and
$U_{p,q}
:=\{ \xi \in \mathcal{N}_{\C} \mid 2^q
 | \xi|_p<1  \}$ for some
$p,q \in \N$, see also
\cite{KSWY98}.
The properties (A1) and (A2) and the
previous remark allow us to state the
following result, which is a special
case of Theorem 8.34 in
\cite{KSWY98}.
\begin{coro}\label{thmStransisomor}
The
$S_{\nu_{\beta}}$-transform is a
topological isomorphism from
$(\mathcal{S})^{-1}_{\nu_{\beta}}$ to
$\text{Hol}_{0}(\mathcal{S}_{\mathbb{C}})$.
\end{coro}

The above characterization result
leads directly to describe the strong
convergence of sequences in
$(\mathcal{S})^{-1}_{\nu_{\beta}}$.

\begin{lem}\label{thmStransConverg}
Let
$\{\Phi_n\}_{n\in \mathbb{N}}$ be a sequence in
$(\mathcal{S})_{\nu_{\beta}}^{-1}$. Then
$\{\Phi_n\}_{n \in \mathbb{N}}$
converges strongly in
$(\mathcal{S})_{\nu_{\beta}}^{-1}$
if and only if there exist
$p,q \in \mathbb{N}$ with the
following two properties:

\begin{itemize}

\item[i)]
$\{S_{\nu_{\beta}}(\Phi_n)
(\xi)\}_{n \in \mathbb{N}}$
is a Cauchy sequence for all
$\xi \in U_{p,q}$;

\item[ii)]
$S_{\nu_{\beta}}(\Phi_n)$ is holomorphic on
$U_{p,q}$ and there is a constant
$C>0$ such that
\[
|S_{\nu_{\beta}}(\Phi_n) (\xi)|\leq C
\]
for all
$\xi \in U_{p,q}$ and for all
$n \in \mathbb{N}$.
\end{itemize}
\end{lem}

\begin{proof}[\bf Proof]
The proof is similar to that of Theorem 2.12 in
\cite{GRO2}.
\end{proof}

\section{Le Roy-Hadamard motion}

In view of the previous results,
we introduce a class of
generalized processes on the space
$
\left ( \mathcal{S}^{\prime
}(\mathbb{R}),\mathcal{B},\nu _{\beta }
\right )
$ defined in
Def.~\ref{cfc}, as a direct
application of the extended dual pairing
to the
function
$^H \mathcal{M}^{\alpha /2}_{-}1_{[0,t)}$, for
$t>0.$

\begin{definition}\label{defprodiff}
Let
$\beta \in (0,1)$,
$\alpha \in (0,2)$
and
$1_{[a,b)}$
 be the indicator function of
$[a,b)$,
 then the \emph{Le Roy-Hadamard motion}
 (hereafter LHm) is defined on the probability space
$
\left ( \mathcal{S}
^{\prime }(\mathbb{R}),\mathcal{B},\nu _{\beta }
\right )
$ as
$B^H_{\alpha, \beta}
:= \{ B^H_{\alpha,
 \beta}(t) \}_{t \geq 0}$, where
\begin{equation}
    B^H_{\alpha , \beta }(t,\omega )
:=\big
 \langle \omega , ^H
 \mathcal{M}^{\alpha /2}_{-}1_{[0,t)}
 \big \rangle ,\qquad
t\geq 0,\;\omega \in
\mathcal{S}^{\prime }(\mathbb{R}).
  \label{i4}
\end{equation}
\end{definition}
\subsection{Finite-dimensional
characterization}
We have that, for any
$t>0$,
$B^H_{\alpha, \beta }
(t,\cdot )\in L^{2}(\nu _{\beta })$
and, by considering
(\ref{mea}), we can write the
$n$-times characteristic function of
$B^H_{\alpha ,\beta }$ as
\begin{equation}
   \Phi _{t_1,...,t_n}
   (\theta_1,...,\theta_n )
:= \mathbb{E}e^{i\sum_{j=1}^{n}\theta_j
 B^H_{\alpha,\beta }(t_j)}=
   \mathcal{R}_{\beta }
\Big \{-\frac{1}{2}\Big \Vert
   \sum_{j=1}^{n}\theta_j
   \mathcal{M}^{\alpha/2}_{-}
   1_{[0,t_j)]}\Big \Vert^2 \Big \},
     \label{char}
\end{equation}
for
$0\leq t_1<t_2<...<t_n$ and
$\theta_j \in \mathbb{R},$ for
$j=1,...,n$. We give the following
characterizations of the process:
\begin{equation}
\left\{B^H_{\alpha,\beta} (t)
\right\}_{t \geq 0}\overset{f.d.d.}{=}
\left\{\sqrt{Y_\beta} B^H_{\alpha}(t)
\right\}_{t \geq 0}\overset{f.d.d.}{=}
\left\{ B^H_{\alpha}
\left (Y_\beta t
\right ) \right\}_{t \geq 0}, \label{C1}
\end{equation}
where
$Y_\beta$, independent of the H-fBm
$B^H_\alpha$, has distribution
$P(Y_{\beta }\in
B)=\mu _{\beta }(B),$ for any
$B\in \mathcal{B}(\mathbb{R}_{+})$.
Indeed,
\begin{equation}
  \int_{\mathbb{R}^+}
  \mathbb{E}e^{i\sum_{j=1}^{n}
  \theta_j \sqrt{y}B^H_{\alpha }(t_j)}d\mu_\beta(y)=
   \int_{\mathbb{R}^+}  \exp
\Big \{-\frac{y}{2}
\Big \Vert \sum_{j=1}^{n}\theta_j
 \mathcal{M}^{\alpha/2}_{-}1_{[0,t_j)]}
 \Big \Vert^2\Big \}d\mu_\beta(y),
   \label{C2}
\end{equation}
which, by considering \eqref{ec},
coincides with \eqref{char}. The last
equality in law follows, by recalling
 the self-similarity property
 (with parameter
$1/2$) of
$B^H_\alpha$, proved in Corollary
 \ref{corostat}.
The two characterizations in
\eqref{C1} are analogous to those
presented for the ggBm, in
\cite{MUR} and
\cite{ERR}, respectively.

It is clear from \eqref{char} that,
in the one-dimensional case, since
$$
 \Vert ^H
\mathcal{M}^{\alpha /2}_{-}1_{[0,t)}
  \Vert ^{2}=t
  $$
(in view of (\ref{der2}) and
(\ref{int2})), the dependence
on the parameter
$\alpha$ is lost, and
\begin{equation}
\Phi _{t}(\theta )
:=\mathbb{E}e^{i\theta
 B^H_{\alpha ,\beta }(t)}
 =\mathcal{R}_{\beta
}
\Big  ( -\frac{\theta ^{2}
\big \Vert ^H
 \mathcal{M}^{\alpha /2}_{-}1_{[0,t)}
 \big \Vert ^{2}}{2}
\Big  ) =
\mathcal{R}_{\beta }
\Big  ( -\frac{\theta ^{2} t}{2}
\Big  ) ,  \label{is}
\end{equation}
 $\theta \in
\mathbb{R},$ $ t \geq 0.$
It follows from (\ref{mea})
and (\ref{mea2})
that
$$
\Phi _{t}(\theta )=
\frac{1}{t}\int_{0}^{\infty }
e^{-\theta ^{2} z/2}m_{\beta
}(z/ t)dz,
$$
so that the following
equality of the one-dimensional
distribution holds, for any
$\alpha,$
\begin{equation}
B^H_{\alpha , \beta }(t)\overset{d}{=}
B(T_{\beta }(t)),\qquad t\geq 0,  \label{dd}
\end{equation}
where
$\left\{ T_{\beta }(t)\right\} _{t\geq 0}$
is a process with
transition density
$g_{\beta }(x,t)=m_{\beta }(x/t)/ t$, for
$x,t\in \mathbb{R
}_{+}$, independent of the standard Brownian motion
$\left\{ B(t)\right\}
_{t\geq 0}.$

It easily follows from (\ref{dd})
 that the process
$B^H_{\alpha ,\beta
}$ has zero mean and
\begin{equation*}
var
\left ( B^H_{\alpha ,\beta }(t)
\right ) =\mathbb{E}T_{\beta }(t)=\frac{1}{t}
\int_{0}^{\infty }zm_{\beta }(z/ t)dz= t,
\end{equation*}
in view of (\ref{ber}),
 regardless of the values of the parameters
$\alpha$ and
$\beta .$

As far as the covariance is concerned,
we can apply Lemma \ref{lem3.2} extended from
$\mathcal{S}(\mathbb{R})$ to
$L^{2}(\mathbb{R})$, so that we have
\begin{align}
cov(B^H_{\alpha , \beta}(t),
B^H_{\alpha,\beta}(s))
&=\int_{\mathcal{S}^{\prime }(\mathbb{R})}
 \langle u,^{H}
 \mathcal{M}_{-}^{\alpha/2 }1_{[0,s)}\rangle
 \langle
u,^{H}\mathcal{M}_{-}^{\alpha/2 }1_{[0,t)}
\rangle d\nu _{\beta }(u)  \notag \\
&=\left
 \langle ^{H}
 \mathcal{M}_{-}^{\alpha/2 }
 1_{[0,s)},^{H}\mathcal{M}
_{-}^{\alpha/2 }1_{[0,t)}
\right\rangle=cov(B^H _{\alpha }
(t),B^H _{\alpha}(s)),  \notag
\end{align}
by \eqref{ss2} and by Theorem \ref{proc}.
Thus, the variance of the process
$B^H_{\alpha,\beta}$ is independent
of both the parameters
$\alpha$ and
$\beta$ and coincides with that of
 the standard Brownian motion,
 while its persistence and memory
 properties are equal to those of
 the Hadamard-fBm (analysed in
 Theorem \ref{cor2.1}) and therefore
  they depend only on
$\alpha$.

Finally, we prove that the
one-dimensional distribution of
$B^H_{\alpha,\beta}$ satisfies a heat
equation with non-constant coefficients
with time-derivative replaced by the
 Hadamard derivative of Caputo
 type of order
$\beta$. This result can be compared
with the master equation, which was
proved in
\cite{MUR} to be satisfied by the
one-dimensional distribution of the
ggBm, and later generalized in
\cite{BEN}.

\begin{theorem}
Let
$\thinspace^{H}D_{0+,t}^{\beta }$
be the  $($left-sided$)$ Hadamard
derivative of Caputo type of order
$\beta \in (0,1),$ defined in
$(\ref{cap}),$ with respect to
$t$. The transition density of
$B^H_{\alpha ,\beta
}$ satisfies, for any
$\alpha,$
the following differential equation
\begin{equation}
^{H}D_{0+,t}^{\beta }u(x,t)=
\frac{ t}{2}\frac{\partial
^{2}}{\partial x^{2}}u(x,t), \label{pde}
\end{equation}
with initial condition
$u(x,0)=\delta (x),$ where
$\delta (\cdot
)$ is the Dirac's delta function.
\end{theorem}

\begin{proof}[\bf Proof]
Observe that
$(^{H}D_{0+}^{\beta }t^{\kappa })(x)
=\kappa ^{\beta
}x^{\kappa }$. Then,
\begin{align}
^{H}D_{0+,z}^{\beta }\mathcal{R}_{\beta }
\left ( sz
\right )  &=  ^{H}D
_{0+,z}^{\beta }\sum_{j=0}^{\infty }
\frac{(sz)^{j}}{(j!)^{\beta }}\,
\label{had2} \\
&=  \frac{1}{\Gamma (1-\beta)}\int_{0}^{z}
\left ( \log \frac{z}{t}
\right ) ^{-\beta }\sum_{j=1}^{\infty }
\frac{js^{j}t^{j-1}}{(j!)^{\beta }}dt,  \notag
\end{align}
where the interchange of derivative
 and series is allowed by the
uniform convergence of the series on
$(0,z)$ (see Theorem 7.17 in
\cite{RUD})$.$ In view of Theorem 7.11 in
\cite{RUD}, applied to
the limit point
$z$ of
$(0,z)$, we can interchange integration and
summation in the last integral:
\begin{align*}
&\lim_{x\rightarrow z}\int_{0}^{x}
\Big ( \log \frac{z}{t}
\Big )
^{-\beta }\lim_{n\rightarrow \infty }
\sum_{j=1}^{n}\frac{js^{j}t^{j-1}}{
(j!)^{\beta }}dt \\
&=\sum_{j=1}^{\infty }\frac{js^{j}}{(j!)^{\beta
}}\lim_{x\rightarrow z}\int_{0}^{x}
\Big ( \log \frac{z}{t}
\Big )
^{-\beta }t^{j-1}dt.
\end{align*}
Indeed, we have uniform convergence
 of the sequence
$\left\{ f_{n}(x)\right\} _{n\geq 1},$
where
\begin{equation*}
f_{n}(x)
:=\int_{0}^{x}
\left ( \log \frac{x}{t}
\right ) ^{-\beta }\sum_{j=1}^{n}
\frac{js^{j}t^{j-1}}{(j!)^{\beta }}dt
=sx\Gamma (1-\beta )\sum_{l=0}^{n-1}
\frac{(sx)^{l}}{(l!)^{\beta }},
\end{equation*}
as
\begin{equation*}
\sup_{0<x<z}\left\vert
f_{n}(x)-f_{m}(x)\right\vert \leq
|s|z\Gamma (1-\beta )
\sum_{l=m}^{n-1}\frac{(|s|z)^{l}}{(l!)^{\beta
}},
\end{equation*}
tends to zero, for
$m,n\rightarrow \infty ,$ by the
convergence of the series to
the Le Roy function.

Therefore, the characteristic
function of the LHm
$\left\{ B_{\alpha ,\beta
}(t)\right\}
_{t\geq 0}$, given in (\ref{is}),
satisfies the equation
\begin{equation}
^{H}D_{0+,t}^{\beta }\widehat{u}(\theta ,t)
=-\frac{ \theta ^{2}t}{2}
\widehat{u}(\theta ,t),  \label{had3}
\end{equation}
with the initial condition
$\widehat{u}(\theta ,0)=1$. Taking the
inverse Fourier transform of (\ref{had3}),
we obtain that the transition density of
$B^H_{\alpha ,\beta
}$ satisfies equation (\ref{pde}) with
$u(x,0)=\delta(x)$.
\end{proof}

\subsection{Le Roy-Hadamard noise}

In order to prove the existence and
an integral formula for the
distributional derivative of the LHm in
$(\mathcal{S})^{-1}_{\nu_{\beta}}$,
we first evaluate  the
$S_{\nu_{\beta}}$-transform of
$B^H_{\alpha,\beta}$, which holds for any
$\alpha \in (0,2)$.

\begin{lem} \label{lem:StrasnfHadRoyBm}
   For
$\alpha \in (0,1) \cup (1,2)$,
$\beta \in (0,1)$ and
$\xi \in \mathcal{S}_\mathbb{C}$, the
$S_{\nu_\beta}$-transform of
$B^H_{\alpha,\beta}$ reads
    \begin{equation}
S_{\nu_{\beta}}(B^H_{\alpha,\beta}(t))(\xi)
= K_\alpha C_{\xi,\beta} t
\Big (\thinspace ^{H}
\mathcal{I}_{0+,1}^{(1+\alpha)/2}\xi
\Big )(t), \qquad \xi \in
 \mathcal{S}_\mathbb{C}, \label{ult}
    \end{equation}
    where
$C_{\xi,\beta}
:=\frac{\mathcal{R}_\beta ' (
 \langle\xi,\xi\rangle/2)}{\mathcal{R}_\beta (
 \langle\xi,\xi\rangle/2)}$,
$K_\alpha = \frac{\Gamma((\alpha +1)/2)}
{\sqrt{\Gamma(\alpha)}}$, and
$\thinspace ^{H}\mathcal{I}_{0+,\mu}^{\nu}$
is the left-sided Hadamard-type integral of order
$\nu >0$ and parameter
$\mu \geq 0$, defined in \eqref{hadint}, and with
$\mathcal{R}_\beta ' (x)
:=\left.\frac{d}{dz}\mathcal{R}_\beta
 (z)\right\vert_{z=x}$.
\end{lem}

\begin{proof}[\bf Proof]
Since
$$
B^H_{\alpha,\beta}(t)=
 \langle \cdot, \thinspace^{H}
 \mathcal{M}_{-}^{\alpha/2 }1_{[0,t)}
 \rangle \in L^2(\nu_\beta) ,
 $$
  the
$S_{\nu_\beta}$-transform is well-defined for
$\xi \in U_{p,q} \subset \mathcal{S}_{\C}$. For
$\omega \in \mathcal{S}'$ and
$s \in [-1,1]$, let
$$
f(\omega,s)
:=\exp(
 \langle \omega, \xi + s
 \thinspace^{H}\mathcal{M}_{-}^{\alpha/2 }
 1_{[0,t)}\rangle) \in L^1(\nu_\beta) .
 $$
  In view of what follows,
$f(\omega,s)$ is differentiable with respect to
$s$, and its derivative is in
$L^1(\nu_\beta)$: indeed we observe that,
for all
$\omega \in \mathcal{S}^{\prime}$,
    \begin{align}
\Big | \frac{d}{ds} f(\omega,s)
\Big |
&\leq |
 \langle \omega, \thinspace^{H}
 \mathcal{M}_{-}^{\alpha/2 }1_{[0,t)}\rangle| \exp(
 \langle \omega, \Re(\xi) \rangle + |
 \langle \omega, \thinspace^{H}
 \mathcal{M}_{-}^{\alpha/2 }1_{[0,t)}\rangle|) \notag\\
        &\leq \exp(
 \langle \omega, \Re(\xi) \rangle + 2|
 \langle \omega, \thinspace^{H}
 \mathcal{M}_{-}^{\alpha/2 }1_{[0,t)}\rangle|)
 \in L^1(\nu_\beta), \label{intdev}
    \end{align}
  by applying the H\"{o}lder inequality, as
$\exp(
 \langle \cdot, \Re(\xi) \rangle)$,
$ \exp(2|
 \langle \cdot, \thinspace^{H}
 \mathcal{M}_{-}^{\alpha/2 }1_{[0,t)}\rangle|)$
  are in
$L^2(\nu_\beta)$, by Lemma \ref{Le1}.
   Finally, by noting that
\[
   \left.\frac{d}{ds}
    f(\omega,s)\right\vert _{s=0}=
 \langle \omega,\thinspace^{H}
 \mathcal{M}_{-}^{\alpha/2 }1_{[0,t)}\rangle e^{
 \langle \omega, \xi \rangle},
\]
   we get
\begin{align*}
S_{\nu_{\beta}}& (B^H_{\alpha,\beta}(t))(\xi)
 =   \frac{1}{\mathcal{R}_\beta(
 \langle \xi,\xi \rangle/2)}
 \int_{\mathcal{S}^\prime}
 \langle \omega,
 \thinspace^{H}
 \mathcal{M}_{-}^{\alpha/2 }1_{[0,t)}\rangle e^{
 \langle \omega, \xi \rangle}
  \nu_\beta(d \omega)\\
 &=
  \frac{1}{\mathcal{R}_\beta(
 \langle \xi,\xi \rangle/2)}
 \int_{\mathcal{S}^\prime}\left.
 \frac{d}{ds} f(\omega,s)
 \right\vert_{ s=0}\nu_\beta(d \omega)\\
       &=  \frac{1}{\mathcal{R}_\beta(
 \langle \xi,\xi \rangle/2)}
 \left.
 \frac{d}{ds}\int_{\mathcal{S}^\prime}
  f(\omega,s)\nu_\beta(d \omega)\right\vert_{ s=0}\\
       &=  \frac{1}{\mathcal{R}_\beta(
 \langle \xi,\xi \rangle/2)}
 \frac{d}{ds}\mathcal{R}_\beta
\Big (\frac{1}{2}
 \langle \xi + s \thinspace^{H}
 \mathcal{M}_{-}^{\alpha/2 }1_{[0,t)},\xi
  + \left. s \thinspace^{H}
  \mathcal{M}_{-}^{\alpha/2 }1_{[0,t)}\rangle
\Big )\right\vert_{ s=0}\\
       &= C_{\xi,\beta}
 \langle \xi, \thinspace^{H}
 \mathcal{M}_{-}^{\alpha/2 }1_{[0,t)}\rangle,
   \end{align*}
   where the interchange of integral and
 derivative is allowed by \eqref{intdev}.
 Formula \eqref{ult} is obtained, by
 \eqref{had4}, \eqref{had6} and
 \eqref{ma}, as follows
    \begin{equation}
S_{\nu_{\beta}}(B^H_{\alpha,\beta}(t))(\xi)
=\frac{C_{\xi,\beta}}
{\sqrt{\Gamma (\alpha)}} \int_0^t \xi(s)
\Big (\log\frac{t}{s}
\Big )^{ { \frac{\alpha-1 }{2} } }
ds<\infty, \label{id}
    \end{equation}
     as any Schwartz function is
      uniformly continuous on
$\mathbb{R}$ and thus belongs to
$AC[0,t],$ for any
$t>0$, and
$\int_0^t
\left (\log\frac{t}{s}
\right )^{ { \frac{\alpha-1 }{2} } }ds<\infty,$ for
$\alpha \in (0,2).$
\end{proof}

\begin{theorem}
    Let
$\alpha\in (0,2)$ and
$\beta \in (0,1)$, then
$B^H_{\alpha,\beta}$ is differentiable in
$(\mathcal{S})^{-1}_{\nu_{\beta}}$
and we define the Le Roy-Hadamard noise as
\begin{equation}\mathscr{N}^{\alpha,\beta}_t
:= \lim_{h\to 0}
\frac{B^H_{\alpha,\beta}(t+h)-
B^H_{\alpha,\beta}(t)}{h}.
\label{eqlimitIntegralcase}
\end{equation}
Moreover, let
\[^{H}
\mathcal{M}_{0+,1}^{\alpha/2}
:=\begin{cases}
K_\alpha \thinspace^{H}
\mathcal{D}_{0+,1}^{ { \frac{1-\alpha}{2} } },
 \qquad \alpha \in (0,1), \\
K_\alpha \thinspace^{H}
\mathcal{I}_{0+,1}^{ { \frac{\alpha-1 }{2} } },
 \qquad \alpha \in (1,2),\end{cases}
\]
where
$\thinspace ^{H}
\mathcal{I}_{0+,\mu}^{\nu}$ $($resp.\
$
\mathcal{D}_{0+,\mu}^{\nu} )$ is the left-sided
Hadamard-type integral $($resp.\ derivative$)$ of order
$\nu >0$ and parameter
$\mu \geq 0$, defined in \eqref{hadint}
 $($resp.\ $\eqref{had});$ then,
we have that, for every
$\xi \in
\mathcal{S}_{\C}$, \begin{equation}
S_{\nu_\beta}(\mathscr{N}^{\alpha,\beta}_t)(\xi)
=C_{\xi,\beta}
\left (\thinspace ^{H}
\mathcal{M}_{0+,1}^{\alpha/2}\xi
\right )(t). \label{eqStrasfDerNoise}
\end{equation}
\end{theorem}

\begin{proof}[\bf Proof]
    Let
$\left\{\Phi_n \right\}_{n\geq 1}$ be defined as
$$
\Phi_n
:=\frac{B^H_{\alpha,\beta}(t+h_n)-
B^H_{\alpha,\beta}(t)}{h_n},
$$
 for
$t\geq 0$ and for a sequence
$\left\{h_n\right\}_{n \geq 1}$ such that
$\lim_{n \to  \infty}h_n =0$.

 i)
 For
$\alpha \in (0,1)$,
$\xi \in
\mathcal{S}_{\mathbb{C}}$,
we have, from \eqref{ult}, that
 \begin{align*}
 & \lim_{n \to \infty}
  S_{\nu_\beta}(\Phi_n (t))(\xi) \\
  &
  =
K_\alpha C_{\xi,\beta}
\lim_{n \to \infty}\frac{1}{h_n}\left[(t+h_n)
\left (\thinspace ^{H}
\mathcal{I}_{0+,1}^{(1+\alpha)/2}\xi
\right )(t+h_n)- t
\left (\thinspace ^{H}
\mathcal{I}_{0+,1}^{(1+\alpha)/2}\xi
\right )(t)\right].
 \end{align*}
By applying  the l'H\^{o}pital rule,
it is then enough to study
 \begin{equation} \lim_{x \to 0}
  \frac{d}{dz} \left.\left[ z
\left (\thinspace ^{H}
\mathcal{I}_{0+,1}^{(1+\alpha)/2}\xi
\right )(z) \right]\right\vert_{z=t+x}
= \lim_{x \to 0}
\Big (\thinspace ^{H}
\mathcal{D}_{0+,1}^{ { \frac{1-\alpha}{2} } }\xi
\Big  )(t+x) .
 \label{id2}
\end{equation}
The existence almost everywhere on
$[0,t]$,
$t>0$, is guaranteed by Lemma 2.34 in
\cite{KIL} and considering that
$\xi \in AC[0,t]$ (see
Remark \ref{rem:AbsoutelycontinousSchwartz}).
In order to derive equation
\eqref{eqStrasfDerNoise},
 we recall the equivalence on
$X_c^p$ between
$^{H}
\mathcal{D}_{0+,\mu}^{\gamma}$
and the left-sided
Marchaud-Hadamard type derivative
$^H\mathbb{D}_{0+,\mu}^\gamma$
 (defined in \eqref{MH2}),
for
$0<\gamma<1$ and
$ \mu \in \mathbb{R}$ (see equation \eqref{MH}).
The definition of Schwartz functions, i.e.,
$$
\sup_{z}
\Big \vert z^k \frac{d^m}{dz^m}\xi(z)
 \Big \vert \leq C_{k,m} ,
 $$
  for any
$k,m \in \mathbb{N}$, ensures that
$
\mathcal{S}
 \subseteq X_c^p$, for any
$p \in [1,\infty)$,
$c>0$, and thus \eqref{MH} is satisfied by
$\xi \in
\mathcal{S}$.  Finally, the continuity of
 the left-sided derivative follows by the
 application of the dominated convergence
 theorem to \eqref{MH2} with
$\gamma= { \frac{1-\alpha}{2} }
$,
$\mu=1$, and allows us to write that
$$
\lim_{x \to 0^+}
\Big  (\thinspace ^{H}
\mathcal{D}_{0+,1}^{ { \frac{1-\alpha}{2} } }\xi
\Big  )(t+x)=
\Big  (\thinspace ^{H}
\mathcal{D}_{0+,1}^{ { \frac{1-\alpha}{2} } }\xi
\Big  )(t) ,
\ \text{ for
$t>0$.}
$$

 ii)
For
$\alpha \in (1,2)$, we can write instead that
 \begin{equation} \lim_{n \to \infty}
 S_{\nu_\beta}(\Phi_n (t))(\xi)
=\frac{(\alpha-1)C_{\xi,\beta}}
{2t \sqrt{\Gamma(\alpha)}} \int_0^t \xi(s)
\Big (\log \frac{t}{s}
\Big )^{(\alpha-3)/2}ds \label{snu}
\end{equation}
and equation \eqref{eqStrasfDerNoise}
follows from \eqref{snu} as
\begin{align}
\lim_{x \to 0} & \frac{1}{x}
\Big [\int_0^{t+x} \xi(s)
\Big  (\log\frac{t+x}{s}
\Big  )^{ { \frac{\alpha-1 }{2} } }ds-
 \int_0^t \xi(s)
\Big  (\log\frac{t}{s}
\Big  )^{ { \frac{\alpha-1 }{2} } }ds
 \Big ] \notag \\
= {} &\lim_{x \to 0} \frac{1}{x}
\int_0^{t} \xi(s) \Big [
\Big  (\log\frac{t+x}{s}
\Big  )^{ { \frac{\alpha-1 }{2} } }-
\Big  (\log\frac{t}{s}
\Big  )^{ { \frac{\alpha-1 }{2} } }
\Big ]ds  \notag \\
&+\lim_{x \to 0} \frac{1}{x}
 \int_t^{t+x} \xi(s)
\Big  (\log\frac{t+x}{s}
\Big  )^{ { \frac{\alpha-1 }{2} } }ds
 \notag \\
= {} &\lim_{x \to 0}
\frac{\alpha -1}{2(t+x)}\Big [\int_0^{t} \!  \xi(s)
\Big  (\log\frac{t}{s}
\Big  )^{(\alpha-3)/2} \! ds+ \!
 \int_t^{t+x}  \! \!  \xi(s)
\Big  (\log\frac{t+x}{s}
\Big  )^{(\alpha-3)/2}ds \Big ].  \notag
\end{align}
Since
$
\mathcal{R}_\beta (\cdot)$ is entire
(see
\cite{SIM}), there are
$p,q \in \mathbb{N}$ and
$K< \infty$ such that
$\left\vert C_{\xi,\beta} \right\vert \leq K$,
for any
$\xi \in U_{p,q}$.

Now, we must prove that
$$
\Big \vert
\Big (\thinspace ^{H}
\mathcal{D}_{+,1}^{ { \frac{1-\alpha}{2} } }\xi
\Big )(t)
\Big \vert \leq C_1
 \ \ \text{ and } \ \
\Big \vert
\Big  (\thinspace ^{H}
\mathcal{I}_{+,1}^{ { \frac{\alpha-1 }{2} } }\xi
\Big  )(t)
\Big \vert \leq C_2,
$$
 for any
$t >0$,
$\xi \in
\mathcal{S}_{\mathbb{C}}(\mathbb{R})$ and for
$C_1 , C_2 >0$.
As far as the integral case is concerned,
by considering the continuity of
$\xi(\cdot)$, we have that for
$\alpha \in (1,2)$ and
$\overline{\xi}
:= \max_{s \in \mathbb{R}} |\xi (s)|$,
$t>0,$
\begin{align*}
\Big \vert\frac{1}{t} \int_0^t \xi(s)
\Big  (\log \frac{t}{s}
\Big  )^{(\alpha-3)/2}ds\Big \vert
&\leq
\overline{\xi}\Big \vert\frac{1}{t} \int_0^t
\Big  (\log \frac{t}{s}
\Big  )^{(\alpha-3)/2}ds\Big \vert \notag \\
&=
\overline{\xi}\int_1^{\infty} \frac{1}{w^2}
\Big  (\log w
\Big  )^{(\alpha-3)/2}dw < \infty.
 \notag
\end{align*}
On the other hand, for the derivative
case and for
$\alpha \in (0,1)$, we resort again to
the equivalence \eqref{MH},
so that we can write
\begin{align*}
\Big \vert
\Big  (\thinspace ^{H}
\mathcal{D}_{+,1}^{ { \frac{1-\alpha}{2} } }\xi
\Big  )(t)
\Big \vert &  =
\Big \vert
\Big  (\thinspace ^{H}
\mathbb{D}_{+,1}^{ { \frac{1-\alpha}{2} } }\xi
\Big  )(t)
\Big \vert  \\
&
\leq
\frac{2\Gamma((\alpha+1)/2)}{1-\alpha}
\Big \vert \frac{1}{t} \int_0^t \frac{\xi(t)-\xi(z)}{
   ( \log\frac{t}{z}
  )^{(3-\alpha)/2}}dz
\Big \vert+\overline{\xi} < \infty \notag.
\end{align*}
Indeed, by considering that
$\log(t/(t-y))>y/t,$ for
$y \in (0,t)$,
\begin{align*}
\Big  \vert \frac{1}{t} \int_0^t
     \frac{\xi(t)-\xi(z)}{
\Big  ( \log\frac{t}{z}
\Big  )^{(3-\alpha)/2}}dz\Big \vert
&\leq
\int_0^t \frac{y}{t}\frac{|\xi(t)-\xi(t-y)|}{y}
\Big  ( \log\frac{t}{t-y}
\Big  )^{(\alpha-3)/2}dy \notag \\
    &\leq  \int_0^t \frac{|\xi(t)
    -\xi(t-y)|}{y}
\Big  ( \log\frac{t}{t-y}
\Big  )^{ { \frac{\alpha-1 }{2} } }dy
 \notag \\
    &\leq   \max_{y \in \mathbb{R}}
    |\xi'(y)|\int_0^t
\Big  ( \log\frac{t}{t-y}
\Big  )^{ { \frac{\alpha-1 }{2} } }dy
< \infty. \notag
\end{align*}
Therefore, for any
$\alpha \in (0,1)\cup (1,2)$ and
$n \in \mathbb{N},$ we have that
$$
 \left\vert
\left (S_{\nu_\beta}(\Phi_n)
\right )(\xi)\right\vert < \infty .
$$
Applying Lemma~\ref{thmStransConverg},
 the sequence
$\left\{\Phi_n \right\}_{n\geq 1}$
converges to some distribution
$\mathscr{N}^{\alpha,\beta}_t$ in
$(
\mathcal{S})^{-1} _{\nu_\beta}$ and
\begin{align*}
 S_{\nu_\beta}(\mathscr{N}^{\alpha,\beta}_t)
(\xi)=\lim_{n \to \infty }
\left (S_{\nu_\beta}(\Phi_n)
\right ) ,
\ \text{ for
$\xi \in U_{p,q}$.}
\tag*{\qedhere}
\end{align*}
\end{proof}

\begin{remark}\rm
   By considering formula \eqref{ult}
   in the limiting case
$\beta=1$, we obtain the
$
\mathcal{S}_{\nu}$-transform of the H-fBm
$B^{H}_{\alpha}$ (where
$\nu=\nu_1$ is the white-noise measure):
\[
 S_{\nu}(B^{H}_{\alpha}(t))(\xi)=K_\alpha t
\left (\thinspace ^{H}
\mathcal{I}_{0+,1}^{(1+\alpha)/2}\xi
\right )(t),  \qquad t>0, \xi \in
\mathcal{S}_{\mathbb{C}}. \notag
\]
    On the other hand, for
$\alpha=\beta=1$, it coincides with
that of the Brownian motion, that is,
$$
S_{\nu_{\beta}}(B^H_{\alpha,\beta}(t))(\xi)
= \int_0^t \xi(s) ds.
$$
  Analogously, for
$\beta=1$, formula \eqref{eqStrasfDerNoise}
gives the
$
\mathcal{S}_{\nu}$-transform of the H-fBm's
 noise, i.e.,
$$
\mathscr{N}^{\alpha}_t
:=\lim_{h \to 0}\frac{B^{H}_{\alpha}(t+h)
-B^{H}_{\alpha}(t)}{h} ,
$$
 which thus reads
$$
S_{\nu}(\mathscr{N}^{\alpha}_t)(\xi)=
\big (\thinspace ^{H}
\mathcal{M}_{0+,1}^{\alpha/2}\xi
\big )(t) ,
$$
 for
$\alpha \in (0,1)\cup (1,2)$,
$t>0$ and
$\xi \in
\mathcal{S}_{\mathbb{C}}$, since
$C_{\xi,1}=1$.
\end{remark}

\subsection{The LH-Ornstein-Uhlenbeck process}
 As for the ggBm, thanks to
 representation \eqref{C1}, the LHm process
 can be considered as a randomly-scaled
  Gaussian process. This allows the
   application of the related
   Ornstein-Uhlenbeck process, described
    below, in modeling physical or
    biological systems. Indeed,
    the environment's
   heterogeneity causes anomalous
    diffusion which could display
    also peculiar memory property, see
\cite{DOVSSPCP19}. If the latter
makes these processes suitable for
depicting complex systems, on the other
hand the study of other theoretical
properties requires an ad hoc
approach (see e.g.
\cite{BBT23}) and deserves further developments.

By means of a procedure similar
to that presented in
\cite{BOC}, we start by defining
 the process
$Y_{\alpha,\beta}
:=\left\{ Y_{\alpha,\beta}(t)
\right\}_{t \geq 0}$,
as the solution to the following
Langevin equation driven by the
LHm (in integral form):
\begin{equation}\label{OU1}
Y_{\alpha,\beta}(t)=y_0-\theta
\int_0^t Y_{\alpha,\beta}(s)ds
+ \sigma B^H_{\alpha,\beta}(t),
 \quad t\geq 0, \end{equation}
    where
$\theta >0$ and
$\sigma \in \mathbb{R}$.
    We now apply the
$
\mathcal{S}_{\nu_\beta}$-transform
to \eqref{OU1} thanks to
  Corollary \ref{thmStransisomor}
  and Lemma \ref{lem:StrasnfHadRoyBm},
   so that, for
$\xi \in U_{p,q}=\{\xi \in
\mathcal{S}_{\mathbb{C}} \mid
2^q \|\xi\|_p^2<1 \}$,
we can write that
    \begin{align}\label{sb}
        \mathcal{S}_{\nu_\beta}
        (Y_{\alpha,\beta}(t))(\xi)
  & =
        y_0-\theta
\mathcal{S}_{\nu_\beta}
\Big (\int_0^t Y_{\alpha,\beta}(s) ds
\Big )(\xi) + \sigma
\mathcal{S}_{\nu_\beta}(
 B^H_{\alpha,\beta}(t))(\xi) \\
        &= y_0-\theta \int_0^t
\mathcal{S}_{\nu_\beta}
(Y_{\alpha,\beta}(s))(\xi) ds
+ \sigma K_\alpha C_{\xi,\beta} t
\Big (\thinspace ^{H}
\mathcal{I}_{0+,1}^{(1+\alpha)/2}\xi
\Big )(t), \notag
    \end{align}
  by applying Theorem~6 in
\cite{KSWY98}. In order to obtain
an ODE solved by the previous
$
\mathcal{S}_{\nu_\beta}$-transform,
we denote the latter as
$y(t)
:=
\mathcal{S}_{\nu_\beta}(Y_{\alpha,\beta}(s))$.
Thus, taking the first derivative
with respect to
$t$ and considering \eqref{id2}, for
$\alpha \in (0,1)$, we have that
\begin{equation}  y'(t)=
-\theta y(t)+\sigma K_\alpha C_{\xi,\beta}
\Big (\thinspace ^{H}
\mathcal{D}_{0+,1}^{ { \frac{1-\alpha}{2} } }\xi
\Big )(t), \quad t> 0, \label{ddt}
     \end{equation}
  with
$y(0)=y_0$. For
$\alpha \in (1,2)$, by recalling \eqref{ult}
 together with \eqref{id}, we have instead that
   \begin{align}   \frac{d}{dt}
\Big [t
\Big (\thinspace ^{H}
\mathcal{I}_{0+,1}^{(1+\alpha)/2}\xi
\Big )(t)\Big ]  & = \frac{1}
{K_\alpha \sqrt{\Gamma(\alpha)}}
\frac{d}{dt}  \int_0^t \xi(s)
\Big (\log\frac{t}{s}
\Big )^{ { \frac{\alpha-1 }{2} } }ds
 \label{sb2} \\
   & =  \frac{1}{\Gamma(
   { \frac{1-\alpha}{2} } ) t}\int_0^t \xi(s)
\Big (\log\frac{t}{s}
\Big )^{(\alpha-3)/2}ds=
\Big (\thinspace ^{H}
\mathcal{I}_{0+,1}^{
{ \frac{\alpha-1 }{2} } }\xi
\Big )(t) \notag,
    \end{align}
so that we obtain, for any
$\alpha \in (0,1) \cup (1,2)$,
\begin{equation}  y'(t)=-\theta y(t)
+\sigma C_{\xi,\beta}
\left (\thinspace ^{H}
\mathcal{M}_{0+,1}^{\alpha/2}\xi
\right )(t), \qquad t> 0. \notag
     \end{equation}

  Solving the above ODE, we write
 \begin{align}
 \label{OU3}
 y(t) & =   y_0e^{-\theta t}
 \\
 &
 \notag
 + \sigma C_{\xi,\beta}
\Big (\int_0^t
\Big (\thinspace ^{H}
\mathcal{M}_{0+,1}^{\alpha/2}\xi
\Big ) (s) ds - \theta \int_0^t
e^{\theta (s-t)} \int_0^s
\Big (\thinspace ^{H}
\mathcal{M}_{0+,1}^{\alpha/2}\xi
\Big ) (u) du \, ds
\Big ) \notag\\
        & =  y_0e^{-\theta t} + \sigma
\Big (
\mathcal{S}_{\nu_\beta}(Y_{\alpha,\beta}(t))(\xi)
 - \theta \int_0^t e^{\theta (s-t)}
\mathcal{S}_{\nu_\beta}(Y_{\alpha,\beta}(s))(\xi) ds
\Big ), \notag
        \end{align}
by taking into account that, by \eqref{sb}
 and \eqref{ddt}, for
$\alpha \in (0,1)$ (resp.\
\eqref{sb} and \eqref{sb2}, for
$\alpha \in (1,2)$),
\begin{equation}\label{de}
\int_0^t
\left (\thinspace ^{H}
\mathcal{M}_{0+,1}^{\alpha/2}\xi
\right ) (s) ds=\frac{1}{C_{\xi,\beta}}
\mathcal{S}_{\nu_\beta}(B^H_{\alpha,\beta}(t))(\xi).
\end{equation}

We now invert the
$
\mathcal{S}_{\nu_\beta}$-transform and
obtain from \eqref{OU3}, for any
$t\geq 0$, the solution to \eqref{de} as
\[Y_{\alpha,\beta}(t)=
y_0e^{-\theta t} + \sigma  B^H_{\alpha,\beta}(t)
- \theta \sigma \int_0^t e^{\theta (s-t)}
 B^H_{\alpha,\beta}(s) ds ,
\]
and we call \textit{LH-Ornstein-Uhlenbeck}
 the process
$Y_{\alpha, \beta}
:=\left\{Y_{\alpha, \beta} (t)\right\}_{t \geq 0}$.

By \eqref{de} and \eqref{id},
\begin{equation}
\int_0^t
\Big (\thinspace ^{H}
\mathcal{M}_{0+,1}^{\alpha/2}\xi
\Big ) (s) ds
= \frac{1}{\sqrt{\Gamma(\alpha)}} \int_0^t \xi(z)
\Big (\log \frac{t}{z}
\Big )^{ { \frac{\alpha-1 }{2} } }dz =
 \langle \xi, \thinspace ^{H}
 \mathcal{M}_{-}^{\alpha/2} 1_{[0,t)}
  \rangle. \notag
\end{equation}
Thus, by defining
\begin{align*}
h^{\alpha, \beta}_t (x)
:   & =
\Big (\thinspace^{H}
\mathcal{M}_{-}^{\alpha/2}1_{[0,t)}
\Big )(x)- \int^{t}_0 e^{\theta(s-t)}
\Big (\thinspace ^{H}
\mathcal{M}_{-}^{\alpha/2}1_{[0,s)}
\Big )(x)ds \notag \\
 & =  \frac{1}{\sqrt{\Gamma(\alpha)}}
\Big [
\Big (\log \frac{t}{x}
\Big )_+^{ { \frac{\alpha-1 }{2} } }-
 \int_x^t  e^{\theta(s-t)}
\Big (\log \frac{s}{x}
\Big )^{ { \frac{\alpha-1 }{2} } }ds
\Big], \notag
\end{align*}
for
$x \in \mathbb{R}^+$,
 we have
\[
Y_{\alpha,\beta}(t)=
y_0e^{-\theta t} + \sigma
 \langle \cdot, h^{\alpha, \beta}_t\rangle,
 \qquad t \geq 0,
\]
so that its characteristic function reads
\[
\mathbb{E}e^{i\sum_{j=1}^{n}
\kappa_j Y_{\alpha,\beta }(t_j)}=
  \exp
\Big  (iy_0 \sum_{j=1}^{n}\kappa_j e^{\theta t_j}
\Big  )
\mathcal{R}_{\beta }
\Big  (-\frac{\sigma^2}{2}
\Big \Vert \sum_{j=1}^{n}\kappa_j
h^{\alpha, \beta}_{t_j}\Big \Vert^2
\Big  ),  \notag
\]
for
$0\leq t_1<t_2<...<t_n$ and
$\kappa_j \in \mathbb{R},$ for
$j=1,...,n$.
Finally, it is easy to check that
$$
\mathbb{E}Y_{\alpha,\beta}(t)
=y_0 e^{-\theta t} ,
\ \text{ for any
$t \geq 0$,}
$$
 and
$$
cov(Y_{\alpha,\beta}(t),
Y_{\alpha,\beta}(s))=\sigma
 \langle h^{\alpha, \beta}_t ,
  h^{\alpha, \beta}_s \rangle ,
\ \text{ for
$s,t \geq 0$.}
$$

\smallskip

\noindent
{\bf Acknowledgments.}
The author F.P.\ thanks the GNAMPA
group of INdAM. The authors L.B.\ and
F.P.\ acknowledge financial support
under the National Recovery and Resilience
Plan (NRRP), Mission 4, Component 2,
Investment 1.1, Call for tender No.\ 104
published on 2.2.2022 by the Italian
 Ministry of University and Research (MUR),
 funded by the European
 Union – NextGenerationEU – Project Title
 ``Non–Markovian Dynamics and Non-local Equations'' –
 202277N5H9 - CUP: D53D23005670006 - Grant Assignment
 Decree No. 973 adopted on June 30, 2023,
 by the Italian Ministry of Ministry of
 University and Research (MUR).

\end{document}